\documentclass[11pt]{article} 
\usepackage{amsfonts,amsmath,latexsym,amssymb,mathrsfs,amsthm,comment}
\usepackage{graphicx}
\usepackage{xcolor}
\usepackage{booktabs}

\usepackage[colorlinks=true, linkcolor=black, citecolor=black]{hyperref}

\numberwithin{equation}{section}

\let\OLDthebibliography\thebibliography
\renewcommand\thebibliography[1]{
  \OLDthebibliography{#1}
  \setlength{\parskip}{1pt}
  \setlength{\itemsep}{0pt plus 0.0ex}
}

\def\numberlikeadb{\global\def\theequation{\thesection.\arabic{equation}}}
\numberlikeadb
\newtheorem{theorem}{Theorem}[section]
\newtheorem{lemma}[theorem]{Lemma}
\newtheorem{corollary}[theorem]{Corollary}

\newtheorem{proposition}[theorem]{Proposition}
\newtheorem{remark}[theorem]{Remark}

\usepackage{lscape}
\usepackage{caption}
\usepackage{multirow}
\allowdisplaybreaks
\begin{document}

\title{The non-central gamma sum and difference distributions: exact distribution and asymptotic expansions
}
\author{Robert E. Gaunt\footnote{Department of Mathematics, The University of Manchester, Oxford Road, Manchester M13 9PL, UK, robert.gaunt@manchester.ac.uk; heather.sutcliffe@manchester.ac.uk}\:\, and Heather L. Sutcliff$\mathrm{e}^{*}$}

\date{} 
\maketitle

\vspace{-5mm}

\begin{abstract} Exact formulas are derived for the probability density functions of the sum and difference of two independent non-central gamma distributed random variables, with both series and integral representations of the density presented. These formulas are then applied to obtain asymptotic expansions for the probability density function, tail probabilities and quantile functions of these distributions. As a special case, we deduce asymptotic expansions for the probability density function of the product of correlated normal random variables with the coefficients given in closed-form. Numerical results are presented to assess the accuracy of our asymptotic approximations across a range of parameter constellations.
\end{abstract}



\section{Introduction}

Following \cite{m93}, we say that the random variable $X$ follows the \emph{non-central gamma} distribution with parameters $\alpha,\beta>0$ and $\lambda\geq0$ if its characteristic function is given by
\begin{align}\label{cf}
\phi_X(t)=\mathbb{E}[\mathrm{e}^{\mathrm{i}tX}]=(1-\mathrm{i}\beta t)^{-\alpha}\exp\bigg(\frac{\mathrm{i}\beta\lambda t}{1-\mathrm{i}\beta t}\bigg), \quad t\in\mathbb{R},   
\end{align}
and we write $X\sim \Gamma(\alpha,\beta,\lambda)$. Here $\alpha$ is a shape parameter, $\beta$ is a rate parameter and $\lambda$ is a non-centrality parameter. When $\lambda=0$, the non-central gamma distribution reduces to the classical gamma distribution, whilst when  $\beta=2$ we recover the non-central chi-square distribution with $2\alpha$ degrees of freedom and non-centrality parameter $2\lambda$ (see \cite[p.\ 436]{jkb95}). The probability density function (PDF) of $X\sim \Gamma(\alpha,\beta,\lambda)$ can be expressed as the infinite series
\begin{align}
f_X(x)=\sum_{k=0}^\infty p_\lambda(k)f_{Y_k}(x),    \label{mix}
\end{align}
where $p_\lambda(k)=\mathrm{e}^{-\lambda}\lambda^k/k!$, $k\geq0$, is the probability mass function of the Poisson distribution with mean $\lambda$ and $Y_k\sim \Gamma(\alpha+k,\beta)$, $k\geq0$, are independent gamma random variables with density $f_{Y_k}(x)=\beta^{-\alpha-k}x^{\alpha+k-1}\mathrm{e}^{-x/\beta}/\Gamma(\alpha+k)$, $x>0$ (see \cite{kb96,p49}). Thus, the non-central gamma distribution can be seen to arise as a Poisson-weighted mixture of gamma distributions. On recalling the power series representation of the modified Bessel function of the first kind $I_\nu(x)$ (see \cite[Eq.\ 10.25.2]{olver}) it can be seen that the sum (\ref{mix}) can be evaluated in closed-form:
\begin{align}\label{pdfncg}
f_X(x)=\frac{1}{\beta}\mathrm{e}^{-\lambda-x/\beta}\bigg(\frac{x}{\beta \lambda}\bigg)^{(\alpha-1)/2}I_{\alpha-1}\bigg(2\sqrt{\frac{\lambda x}{\beta}}\bigg), \quad x>0;  
\end{align}
see also \cite[Theorem 1.3.4]{m05} for a similar closed-form representation for the PDF of the non-central chi-square distribution.

Sums and differences of independent non-central gamma random variables arise in numerous settings throughout the mathematical sciences. For instance, quadratic forms in normal random variables can be expressed as linear combinations of independent non-central chi-square random variables (see, for example, \cite{s63}). The problem of obtaining the exact distribution of such random variables has received much interest in the statistics literature (see, for example, \cite{cl05a,cl05,g55,i61,k67a,k67,p66,s63,sk61}) and applications have been found in diverse areas, such as radio propagation \cite{l23}, counting string vacua \cite{b13} and uncertainty quantification of turbulent flows \cite{m18}, and some of the numerous applications in mathematical statistics are compiled in \cite[Chapter 7]{mp92}. 

The sum and difference of two independent central and non-central gamma random variables, which we will refer to as the (non-central) gamma sum and difference distributions, have also received interest in the literature. In the case of a common shape parameter $\alpha$, the (central) gamma sum and difference distributions were shown by \cite{ha04} to follow the McKay Type I and II distributions \cite{m32}, respectively, and this result was applied to a problem in fading channel theory. The exact distributional theory of the McKay Type I distribution has received recent interest \cite{pog1,pog2}, whilst the McKay Type II distribution is also referred to as the variance-gamma distribution \cite{vg review}, which is widely used in financial modelling \cite{madan}. Again in the case of a common shape parameter (degrees of freedom), the difference of two independent non-central chi-square random variables is equal in distribution to the product of correlated normal random variables, and more generally the sum of independent copies of such random variables \cite{g25,j26}. These distributions have received attention in the statistics literature since the work of \cite{c32,w32} in the 1930's, and have numerous applications throughout the applied sciences, some of which are listed in \cite{g22,np16}. A review of the (central) gamma difference distribution is given in \cite{k15}, and recent contributions that consider different shape parameters include \cite{f23,h21}. 

In this paper, we study in detail the non-central gamma sum and difference distributions, that is the distributions of the random variables $X_1+X_2$ and $X_1-X_2$, where $X_1\sim \Gamma(\alpha_1,\beta_1,\lambda_1)$ and $X_2\sim \Gamma(\alpha_2,\beta_2,\lambda_2)$ are independent non-central gamma random variables. To fix notation, we shall say that $Y\sim \mathrm{NCGS}(\alpha_1,\alpha_2,\beta_1,\beta_2,\lambda_1,\lambda_2)$ if $Y=_d X_1+X_2$ (where $=_d$ denotes equality in distribution) and say that $Z\sim \mathrm{NCGD}(\alpha_1,\alpha_2,\beta_1,\beta_2,\lambda_1,\lambda_2)$ if $Z=_d X_1-X_2$. Here, in a manner that is consistent with the notation of \cite{j26} for the weighted non-central chi-square difference distribution, we have written NCGS for ``Non-Central Gamma Sum" distribution and NCGD for ``Non-Central Gamma Difference" distribution. 

In Section \ref{sec2}, we provide exact expressions for the PDF of the non-central gamma sum and difference distributions, providing both series representations (Section \ref{sec2.1}) and integral representations (Section \ref{sec2.2}). For the general case $\alpha_i,\beta_i>0$ and $\lambda_i\geq0$ for $i=1,2$,  our series representations for the PDF of $X_1-X_2$ and $X_1+X_2$ take the form of double infinite series expressed in terms of confluent hypergeometric functions (Theorem \ref{thm1.1}). A number of simplifications for specific parameter values are given in Remark \ref{zero}, Proposition \ref{prop1} and Theorem \ref{thm2.3}, and in some cases we recover known formulas from the literature. As mentioned earlier, exact (but rather involved) formulas for the PDF of linear combinations of non-central chi-square random variables are already available in the literature; for example, infinite series representations expressed in terms of generalized Laguerre polynomials are given in \cite{cl05} and references therein. In the case of a weighted sum of two non-central chi-square random variables, our formulas for the PDF offer a simpler alternative to these formulas, which, typically involve more complicated coefficients and require an additional parameter to be selected to ensure convergence of the series. The simpler form offered by our formulas allows one to more easily gain insight into the non-central gamma sum and difference distributions, as is exhibited in Section \ref{sec3}.

Since our exact formulas for the PDF of the non-central gamma sum and difference distributions (and other exact formulas in the case of linear combinations of two non-central chi-square random variables) take a rather complicated form, it is of interest to study the asymptotic behaviour of these distributions. This is the subject of Section \ref{sec3}, and our work can be seen to complement other studies, such as \cite{gz25,seg0,seg1,seg2,t93,quantilepaper,erlang}, that concern the asymptotic behaviour of important probability distributions. We begin in Propositions \ref{prop2.4} and \ref{prop2.5} by studying the asymptotic behaviour of the PDFs $f_{X_1-X_2}(x)$ and $f_{X_1+X_2}(x)$ of the difference $X_1-X_2$ and the sum $X_1+X_2$ as $x\rightarrow0$. This study leads to insight into the asymptotic behaviour of the distributions at the origin and also allows us to infer results concerning unimodality and boundedness of the distributions. 

In Theorem \ref{thm3.1}, we move on to study the asymptotic behaviour of the PDFs as $x\rightarrow\infty$ (and also as $x\rightarrow-\infty$ in the case of the difference distribution), obtaining the entire asymptotic expansion. Here the convenient form of our series representations (\ref{diff}) and (\ref{sum}) for the PDFs of $X_1-X_2$ and $X_1+X_2$ allows for the entire asymptotic expansion to be derived with simple explicit formulas given for the coefficients, expressed in terms of generalized Laguerre polynomials. 
From the expansions of Theorem \ref{thm3.1}, we deduce the entire asymptotic expansions for the tail probabilities of the non-central gamma sum and difference distributions for large $x$ (Theorem \ref{thm3.5}), and in turn we apply these expansions to deduce asymptotic approximations for the quantile function of these distributions (Theorem \ref{thm3.7}). 

Given that a special case of the non-central gamma difference distribution is the product of correlated normal random variables, and more generally the sum of independent copies of such random variables, an important special case of our results concerns these distributions. In Corollary \ref{corpn}, we provide asymptotic expansions for the PDF of the sum of independent copies of the product of correlated normal random variables as a special case of Theorem \ref{thm3.1}. Corollary \ref{corpn} offers a refinement of Theorem 3.1 of \cite{gz25}, in which the entire asymptotic expansion was derived for the PDF as $x\rightarrow\pm\infty$; however, the coefficients were more difficult to compute as they were given in terms of coefficients of a Puiseux series of a certain function of two variables. On the other hand, our coefficients are fully explicit and simple to compute. In Corollary \ref{corpn2}, we provide simpler formulas still for the product of correlated normal random variables, with coefficients expressed in terms of Hermite polynomials.

The asymptotic approximations given in Section \ref{sec3} for the PDF, tail probabilities and quantile function of the non-central gamma sum and difference distributions are very simple to implement and computationally efficient. In Section \ref{sec4}, we present numerical results that demonstrate that even using just the first few terms in the asymptotic expansions offer reasonable accuracy across a range of parameter constellations, and that retaining further terms in the asymptotic expansion for the PDF can lead to highly accurate approximations.





\section{Exact formulas for the density}\label{sec2}

\subsection{Series representations for the density}\label{sec2.1}

Let $X_1\sim \Gamma(\alpha_1,\beta_1,\lambda_1)$ and $X_2\sim \Gamma(\alpha_2,\beta_2,\lambda_2)$ be independent non-central gamma random variables, where $\alpha_i,\beta_i>0$ and $\lambda_i\geq0$ for $i=1,2$. In the following theorem, we provide series representations for the PDF of $X_1-X_2$ and $X_1+X_2$, the difference and the sum of independent non-central gamma random variables. The proof is given in Section \ref{sec5}, as are all other proofs in this paper. The formulas given in Theorem \ref{thm1.1} are valid for all $\alpha_i,\beta_i>0$ and $\lambda_i\geq0$ for $i=1,2$. The formulas are expressed in terms of the confluent hypergeometric function of the first kind $M(a,b,x)$ and the confluent hypergeometric function of the second kind $U(a,b,x)$; see \cite[Chapter 13]{olver} for definitions and standard properties. In order to state our formulas, we introduce the following notation. We let $b(x)=\beta_1$ and $a_{j,k}(x)=\alpha_1+k-j$ for $x>0$, and $b(x)=\beta_2$ and $a_{j,k}(x)=\alpha_2+j$ for $x<0$. 

\begin{theorem}\label{thm1.1} For $x\in\mathbb{R}$,
\begin{align}
		f_{X_1-X_2}(x)&=\beta_1^{-\alpha_1}\beta_2^{-\alpha_2}\mathrm{e}^{-\lambda_1-\lambda_2-|x|/b(x)}\sum_{k=0}^{\infty}\sum_{j=0}^{k}\frac{1}{k!\Gamma(a_{j,k}(x))}\binom{k}{j}\bigg(\frac{\lambda_1}{\beta_1}\bigg)^k\bigg(\frac{\lambda_2\beta_1}{\lambda_1\beta_2}\bigg)^j \nonumber\\
		&\quad\times\bigg(\frac{1}{\beta_1}+\frac{1}{\beta_2}\bigg)^{1-\alpha_1-\alpha_2-k}U\bigg(1-a_{j,k}(x),2-\alpha_1-\alpha_2-k,\bigg(\frac{1}{\beta_1}+\frac{1}{\beta_2}\bigg)|x|\bigg), \label{diff}
	\end{align}
and, for $x>0$,
\begin{align}
	f_{X_1+X_2}(x)&=\beta_1^{-\alpha_1}\beta_2^{-\alpha_2}\mathrm{e}^{-\lambda_1-\lambda_2-x/\beta_1}x^{\alpha_1+\alpha_2-1}\sum_{k=0}^{\infty}\frac{1}{k!\Gamma(\alpha_1+\alpha_2+k)}\bigg(\frac{\lambda_1x}{\beta_1}\bigg)^k\nonumber\\
	&\quad\times \sum_{j=0}^{k}\binom{k}{j}\bigg(\frac{\lambda_2\beta_1}{\lambda_1\beta_2}\bigg)^j M\bigg(\alpha_2+j,\alpha_1+\alpha_2+k,\bigg(\frac{1}{\beta_1}-\frac{1}{\beta_2}\bigg)x\bigg).\label{sum}
\end{align}
\end{theorem}

\begin{remark}\label{ffrr}
Given the interest in linear combinations of independent non-central chi-square random variables, we provide the following distributional relations to allow readers to easily translate the formulas of Theorem \ref{thm1.1} and all other expressions in this paper from the non-central gamma sum and difference setting to linear combinations of two non-central chi-square random variables. For independent non-central chi-square random variables $V_1\sim \chi_{r_1}'^2(\lambda_1)$ and $V_2\sim \chi_{r_2}'^2(\lambda_2)$, if $w_1,w_2>0$, we have 
\begin{align*}w_1V_1+w_2V_2\sim \mathrm{NCGS}(r_1/2,r_2/2,2w_1,2w_2,\lambda_1/2,\lambda_2/2)
\end{align*}
and  
\begin{align}\label{chirep}
w_1V_1-w_2V_2\sim \mathrm{NCGD}(r_1/2,r_2/2,2w_1,2w_2,\lambda_1/2,\lambda_2/2).
\end{align}
On this matter, we also note that formula (\ref{diff}) is a generalisation of formula (2.5) of \cite{g25} for the PDF of the difference of two independent non-central chi-square random variables with the same number of degrees of freedom.
\end{remark}


\begin{remark}\label{zero} 1. Since $M(a,b,0)=1$ for all $a,b>0$ (see \cite[Eq.\ 13.2.2]{olver}), one can immediately see that the series representation (\ref{sum}) for the PDF of the sum $X_1+X_2$ simplifies when $\beta_1=\beta_2=\beta$. Indeed, in this case, 
\[X_1+X_2\sim \Gamma(\alpha_1+\alpha_2,\beta,\lambda_1+\lambda_2),\] 
which can be easily deduced by computing the characteristic function of $X_1+X_2$ using formula (\ref{cf}).   

\vspace{3mm}

\noindent 2. When $\lambda_1=0$ the PDF (\ref{diff}) reduces to the following single infinite series:
	\begin{align}\label{diff0}
		f_{X_1-X_2}(x)&=\beta_1^{-\alpha_1}\beta_2^{-\alpha_2}\mathrm{e}^{-\lambda_2-|x|/b(x)}\sum_{k=0}^{\infty}\frac{1}{k!\Gamma(a_{k,k}(x))}\bigg(\frac{\lambda_2}{\beta_2}\bigg)^k\bigg(\frac{1}{\beta_1}+\frac{1}{\beta_2}\bigg)^{1-\alpha_1-\alpha_2-k} \nonumber\\
		&\quad\times U\bigg(1-a_{k,k}(x),2-\alpha_1-\alpha_2-k,\bigg(\frac{1}{\beta_1}+\frac{1}{\beta_2}\bigg)|x|\bigg), \quad x\in\mathbb{R},
	\end{align}
whilst when $\lambda_2=0$ we get
\begin{align}
			f_{X_1-X_2}(x)&=\beta_1^{-\alpha_1}\beta_2^{-\alpha_2}\mathrm{e}^{-\lambda_1-|x|/b(x)}\sum_{k=0}^{\infty}\frac{1}{k!\Gamma(a_{0,k}(x))}\bigg(\frac{\lambda_1}{\beta_1}\bigg)^k\bigg(\frac{1}{\beta_1}+\frac{1}{\beta_2}\bigg)^{1-\alpha_1-\alpha_2-k}\nonumber\\
		&\quad\times U\bigg(1-a_{0,k}(x),2-\alpha_1-\alpha_2-k,\bigg(\frac{1}{\beta_1}+\frac{1}{\beta_2}\bigg)|x|\bigg), \quad x\in\mathbb{R}. \nonumber
\end{align}
When $\lambda_1=\lambda_2=0$ the PDF reduces to a single term:
\begin{align}
		f_{X_1-X_2}(x)&=\frac{1}{\Gamma(a_{0,0}(x))}\beta_1^{-\alpha_1}\beta_2^{-\alpha_2}\bigg(\frac{1}{\beta_1}+\frac{1}{\beta_2}\bigg)^{1-\alpha_1-\alpha_2}\mathrm{e}^{-|x|/b(x)}\nonumber\\
		&\quad\times U\bigg(1-a_{0,0}(x),2-\alpha_1-\alpha_2,\bigg(\frac{1}{\beta_1}+\frac{1}{\beta_2}\bigg)|x|\bigg), \quad x\in\mathbb{R}, \label{blond}
	\end{align}
which is in agreement with the formulas for the PDF of the (central) gamma difference distribution given by \cite[Eq.\ (23)]{h17} and \cite[Eq.\ (2.3)]{h21}.    

If we further assume that $\alpha_1=\alpha_2=\alpha$, then the PDF of $X_1-X_2$ can be expressed in terms of the modified Bessel function of the second kind $K_\nu(x)$ (see \cite[Chapter 10]{olver}). Indeed, applying the relation $U(a,2a,2x)=\pi^{-1/2}\mathrm{e}^x(2x)^{1/2-a}{K}_{a-1/2}(x)$ (see \cite[Eq.\ 13.6.10]{olver}) to the formula (\ref{blond}) yields the expression
\begin{align}
f_{X_1-X_2}(x)&=\frac{(\beta_1\beta_2)^{-\alpha}}{\sqrt{\pi}\Gamma(\alpha)}\bigg(\frac{1}{\beta_1}+\frac{1}{\beta_2}\bigg)^{1/2-\alpha}\exp\bigg\{\bigg(\frac{1}{\beta_2}-\frac{1}{\beta_1}\bigg)\frac{x}{2}\bigg\}\nonumber\\
&\quad\times|x|^{\alpha-1/2}K_{\alpha-1/2}\bigg(\bigg(\frac{1}{\beta_1}+\frac{1}{\beta_2}\bigg)\frac{|x|}{2}\bigg),\quad x\in\mathbb{R}.   \label{vgpdf}
\end{align}
We recognise (\ref{vgpdf}) as the PDF of a variance-gamma random variable, which is to be expected since the difference of two independent gamma random variables with common shape parameter is variance-gamma distributed (see \cite[Section 2.4]{vg review}).

\vspace{3mm}

\noindent 3. When $\lambda_1=0$ the PDF (\ref{sum}) reduces to the following single infinite series:
	\begin{align}\label{sum2}
		f_{X_1+X_2}(x)&=\beta_1^{-\alpha_1}\beta_2^{-\alpha_2}\mathrm{e}^{-\lambda_2-x/\beta_1}x^{\alpha_1+\alpha_2-1}\sum_{k=0}^{\infty}\frac{1}{k!\Gamma(\alpha_1+\alpha_2+k)}\bigg(\frac{\lambda_2x}{\beta_2}\bigg)^k\nonumber\\
	&\quad\times  M\bigg(\alpha_2+k,\alpha_1+\alpha_2+k,\bigg(\frac{1}{\beta_1}-\frac{1}{\beta_2}\bigg)x\bigg), \quad x>0,
	\end{align}
and when $\lambda_2=0$ the PDF (\ref{sum}) reduces to
\begin{align}
	f_{X_1+X_2}(x)&=\beta_1^{-\alpha_1}\beta_2^{-\alpha_2}\mathrm{e}^{-\lambda_1-x/\beta_1}x^{\alpha_1+\alpha_2-1}\sum_{k=0}^{\infty}\frac{1}{k!\Gamma(\alpha_1+\alpha_2+k)}\bigg(\frac{\lambda_1x}{\beta_1}\bigg)^k\nonumber\\
	&\quad\times  M\bigg(\alpha_2,\alpha_1+\alpha_2+k,\bigg(\frac{1}{\beta_1}-\frac{1}{\beta_2}\bigg)x\bigg), \quad x>0.\nonumber
\end{align}
When $\lambda_1=\lambda_2=0$ the PDF(\ref{sum}) reduces to a single term:  
\begin{align}\label{rj}
		f_{X_1+X_2}(x)&=\frac{\beta_1^{-\alpha_1}\beta_2^{-\alpha_2}}{\Gamma(\alpha_1+\alpha_2)}x^{\alpha_1+\alpha_2-1}\mathrm{e}^{-x/\beta_1} M\bigg(\alpha_2, \alpha_1+\alpha_2, \bigg(\frac{1}{\beta_1}-\frac{1}{\beta_2}\bigg)x\bigg), \quad x>0.
\end{align}

If we further assume that $\alpha_1=\alpha_2=\alpha$, the PDF of $X_1+X_2$ can be expressed in terms of the modified Bessel function of the first kind. Applying the formula $M(a,2a,2x)=\Gamma(a+1/2)\mathrm{e}^x(x/2)^{1/2-a}I_{a-1/2}(x)$ to equation (\ref{rj}) (see \cite[Eq.\ 13.6.9]{olver}) gives that, for $\beta_1\not=\beta_2$,
\begin{align}
f_{X_1+X_2}(x)&=\frac{\Gamma(\alpha+1/2)}{\Gamma(2\alpha)}(\beta_1\beta_2)^{-\alpha}\bigg|\frac{1}{\beta_1}-\frac{1}{\beta_2}\bigg|^{1/2-\alpha}\nonumber\\
&\quad\times (2x)^{\alpha-1/2}\exp\bigg\{-\bigg(\frac{1}{\beta_1}+\frac{1}{\beta_2}\bigg)\frac{x}{2}\bigg\}I_{\alpha-1/2}\bigg(\bigg|\frac{1}{\beta_1}-\frac{1}{\beta_2}\bigg|\frac{x}{2}\bigg), \quad x>0, \label{7337}   
\end{align}
which we recognise as the PDF of a McKay Type I random variable, as would be expected given that it was shown by \cite{ha04} that the sum of two independent gamma random variables follows the McKay Type I distribution.
In obtaining (\ref{7337}) we used the identity $x^\nu I_\nu(x)=|x|^\nu I_\nu(|x|)$ for $x\in\mathbb{R}$ (which is easily inferred from the power series representation of the modified Bessel function of the first kind given by \cite[Eq.\ 10.25.2]{olver}) in order to write down an expression in which all terms are real-valued. When $\beta_1=\beta_2=\beta$, by applying the limiting form $I_\nu(x)\sim (x/2)^\nu/\Gamma(\nu+1)$, as $x\rightarrow0$ (for $\nu>-1$)
(which is an immediate consequence of the power series representation of the modified Bessel function of the first kind given by \cite[Eq.\ 10.25.2]{olver}), the expression (\ref{7337}) is seen to reduce to the PDF of the $\Gamma(2\alpha,\beta)$ distribution.
\end{remark}

In the following proposition, we note some further cases in which the infinite series representations of the PDF of the difference $X_1-X_2$ reduce to single finite series. The formulas are expressed in terms of the generalized Laguerre polynomials $L_n^{(\alpha)}(x)$ (see \cite[Chapter 18]{olver}) and the Pochhammer symbol $(x)_n=x(x+1)\cdots(x+n-1)$.

\begin{proposition}\label{prop1} 1. Suppose that $\lambda_1=0$ and $\alpha_1=n$ is a positive integer. Then, for $x>0$, 
\begin{align}\label{rprp}
f_{X_1-X_2}(x)=\frac{\beta_1^{-n}\beta_2^{-\alpha_2}}{(n-1)!}\bigg(\frac{1}{\beta_1}+\frac{1}{\beta_2}\bigg)^{-\alpha_2}x^{n-1}\exp\bigg(-\frac{x}{\beta_1}-\frac{\lambda_2\beta_2}{\beta_1+\beta_2}\bigg)\sum_{k=0}^{n-1}\frac{u_k(n,\alpha_2,\beta_1,\beta_2,\lambda_2)}{x^k},    
\end{align}  
where $u_0(n,\alpha_2,\beta_1,\beta_2,\lambda_2)=1$ and, for $1\leq k\leq n-1$,
\begin{align*}
u_k(n,\alpha_2,\beta_1,\beta_2,\lambda_2)&=(-1)^k(1-n)_k\bigg(\frac{1}{\beta_1}+\frac{1}{\beta_2}\bigg)^{-k}L_k^{(\alpha_2-1)}\bigg(-\frac{\lambda_2\beta_1}{\beta_1+\beta_2}\bigg).
\end{align*}
2. Suppose that $\lambda_2=0$ and $\alpha_2=n$ is a positive integer. Then, for $x<0$, 
\begin{align}\label{rprp2}
f_{X_1-X_2}(x)=\frac{\beta_1^{-\alpha_1}\beta_2^{-n}}{(n-1)!}\bigg(\frac{1}{\beta_1}+\frac{1}{\beta_2}\bigg)^{-\alpha_1}|x|^{n-1}\exp\bigg(\frac{x}{\beta_2}-\frac{\lambda_1\beta_1}{\beta_1+\beta_2}\bigg)\sum_{k=0}^{n-1}\frac{u_k(n,\alpha_1,\beta_2,\beta_1,\lambda_1)}{|x|^k}.    
\end{align}
\end{proposition}

\begin{remark}
Part 1 of Proposition \ref{prop1} provides a closed-form formula for the PDF $f_{X_1-X_2}(x)$ for $x>0$, where $X_1\sim \Gamma(n,\beta_1,0)$ and $X_2\sim \Gamma(\alpha_2,\beta_2,\lambda_2)$ are independent. In this case, $X_1\sim \Gamma(n,\beta_1)$ (since $\lambda_1=0$) is distributed as the sum of $n$ independent exponential random variables with mean $\beta_1$. A similar comment applies to part 2 of the proposition.    
\end{remark}

In the case $\lambda_1=\lambda_2=\lambda$ and $\beta_1=\beta_2=\beta$ we are able to obtain a representation of the PDF of the difference $X_1-X_2$ expressed as a single infinite series involving the modified Bessel function of the second kind.
    
\begin{theorem}\label{thm2.3}
 Suppose that $\lambda_1=\lambda_2=\lambda$ and $\beta_1=\beta_2=\beta$. Then, for $x\in\mathbb{R}$,
	\begin{align}\label{thm2.31}
		f_{X_1-X_2}(x)=\frac{\mathrm{e}^{-2\lambda}}{\beta\sqrt{\pi}}\sum_{k=0}^{\infty}\frac{(2\lambda)^k}{k!\Gamma(\alpha_1+\alpha_2+k)}\bigg(\frac{|x|}{2\beta}\bigg)^{\alpha_1+\alpha_2+k-1/2}K_{\alpha_1+\alpha_2+k-1/2}\bigg(\frac{|x|}{\beta}\bigg).
	\end{align}
\end{theorem}

\subsection{Integral representations for the density}\label{sec2.2}

In the following theorem, we present integral representations for the PDFs of the non-central gamma difference and sum distributions.

\begin{theorem}\label{thm4.1} 1. Suppose that $\lambda_1,\lambda_2\not=0$. Then, for $x>0$,
	\begin{align}
	f_{X_1-X_2}(x)&=D_1(x)\int_{0}^{\infty}t^{(\alpha_2-1)/2}(1+t)^{(\alpha_1-1)/2}\exp\bigg(-\bigg(\frac{1}{\beta_1}+\frac{1}{\beta_2}\bigg)xt\bigg)\nonumber\\
	&\quad\times I_{\alpha_1-1}\bigg(2\sqrt{\frac{\lambda_1x(1+t)}{\beta_1}}\bigg)I_{\alpha_2-1}\bigg(2\sqrt{\frac{\lambda_2xt}{\beta_2}}\bigg)\,\mathrm{d}t,\label{diffint1}\\
	f_{X_1+X_2}(x)&=D_1(x)\int_{0}^{1}t^{(\alpha_2-1)/2}(1-t)^{(\alpha_1-1)/2}\exp\bigg(\bigg(\frac{1}{\beta_1}-\frac{1}{\beta_2}\bigg)xt\bigg)\nonumber\\
	&\quad\times I_{\alpha_1-1}\bigg(2\sqrt{\frac{\lambda_1x(1-t)}{\beta_1}}\bigg)I_{\alpha_2-1}\bigg(2\sqrt{\frac{\lambda_2xt}{\beta_2}}\bigg)\,\mathrm{d}t,\label{sumint1}
	\end{align}
	where
	\begin{align}
		D_1(x)=\bigg(\frac{\lambda_1}{\beta_1}\bigg)^{(1-\alpha_1)/2}\bigg(\frac{\lambda_2}{\beta_2}\bigg)^{(1-\alpha_2)/2}\beta_1^{-\alpha_1}\beta_2^{-\alpha_2}x^{(\alpha_1+\alpha_2)/2}\mathrm{e}^{-\lambda_1-\lambda_2-x/\beta_1}\nonumber.
	\end{align}
	2. Suppose that $\lambda_1=0$ and $\lambda_2\not=0$. Then, for $x>0$,
	\begin{align}
		f_{X_1-X_2}(x)&=D_2(x)\int_{0}^{\infty}t^{(\alpha_2-1)/2}(1+t)^{\alpha_1-1}\exp\bigg(-\bigg(\frac{1}{\beta_1}+\frac{1}{\beta_2}\bigg)xt\bigg)I_{\alpha_2-1}\bigg(2\sqrt{\frac{\lambda_2xt}{\beta_2}}\bigg)\,\mathrm{d}t,\label{diffint2}\\
	f_{X_1+X_2}(x)&=D_2(x)\int_{0}^{1}t^{(\alpha_2-1)/2}(1-t)^{\alpha_1-1}\exp\bigg(\bigg(\frac{1}{\beta_1}-\frac{1}{\beta_2}\bigg)xt\bigg)I_{\alpha_2-1}\bigg(2\sqrt{\frac{\lambda_2xt}{\beta_2}}\bigg)\,\mathrm{d}t,\nonumber
	\end{align}
	where 
	\begin{align}
		D_2(x)&=\frac{1}{\Gamma(\alpha_1)}\bigg(\frac{\lambda_2}{\beta_2}\bigg)^{(1-\alpha_2)/2}\beta_1^{-\alpha_1}\beta_2^{-\alpha_2}x^{(2\alpha_1+\alpha_2-1)/2}\mathrm{e}^{-\lambda_2-x/\beta_1}\nonumber.
	\end{align}
	3. Suppose that $\lambda_2=0$ and $\lambda_1\not=0$. Then, for $x>0$,
		\begin{align}
		f_{X_1-X_2}(x)&=D_3(x)\int_{0}^{\infty}t^{\alpha_2-1}(1+t)^{(\alpha_1-1)/2}\exp\bigg(-\bigg(\frac{1}{\beta_1}+\frac{1}{\beta_2}\bigg)xt\bigg)I_{\alpha_1-1}\bigg(2\sqrt{\frac{\lambda_1x(1+t)}{\beta_1}}\bigg)\,\mathrm{d}t,\label{diffint3}\\
		f_{X_1+X_2}(x)&=D_3(x)\int_{0}^{1}t^{\alpha_2-1}(1-t)^{(\alpha_1-1)/2}\exp\bigg(\bigg(\frac{1}{\beta_1}-\frac{1}{\beta_2}\bigg)xt\bigg)I_{\alpha_1-1}\bigg(2\sqrt{\frac{\lambda_1x(1-t)}{\beta_1}}\bigg)\,\mathrm{d}t,\nonumber
	\end{align}
	where 
	\begin{align}
		D_3(x)&=\frac{1}{\Gamma(\alpha_2)}\bigg(\frac{\lambda_1}{\beta_1}\bigg)^{(1-\alpha_1)/2}\beta_1^{-\alpha_1}\beta_2^{-\alpha_2}x^{(\alpha_1+2\alpha_2-1)/2}\mathrm{e}^{-\lambda_1-x/\beta_1}\nonumber.
	\end{align}
	4. Integral representations for the PDF $f_{X_1-X_2}(x)$ that hold for $x<0$ are obtained by replacing $(x,\alpha_1,\alpha_2,\beta_1,\beta_2,\lambda_1,\lambda_2)$ by $(-x,\alpha_2,\alpha_1,\beta_2,\beta_1,\lambda_2,\lambda_1)$ in formulas (\ref{diffint1}), (\ref{diffint2}) and (\ref{diffint3}). 
\end{theorem}

\begin{remark}\label{remdiff0}
The assertion in part 4 of Theorem \ref{thm4.1} is immediate from the basic distributional relation $X_1-X_2=_d -(X_1'-X_2')$, where $X_1'\sim\Gamma(\alpha_2,\beta_2,\lambda_2)$ and $X_2'\sim\Gamma(\alpha_1,\beta_1,\lambda_1)$ are independent.    
\end{remark}


In the case $\lambda_1=\lambda_2=\lambda$ and $\beta_1=\beta_2=\beta$ we can obtain a simpler integral representation for the PDF of the non-central gamma difference distribution.

\begin{theorem}
	\label{thm4.2} Suppose that $\lambda_1=\lambda_2=\lambda$ and $\beta_1=\beta_2=\beta$. Then, for $x\in\mathbb{R},$
	\begin{align}\label{thm4.21}
		f_{X_1-X_2}(x)&=\frac{\mathrm{e}^{-2\lambda}}{2^{\alpha_1+\alpha_2}\beta\sqrt{\pi}}\bigg(\frac{\lambda}{2}\bigg)^{(1-\alpha_1-\alpha_2)/2}\bigg(\frac{|x|}{\beta}\bigg)^{\alpha_1+\alpha_2}\nonumber\\
		&\quad\times\int_{0}^{\infty}t^{-(\alpha_1+\alpha_2+2)/2}\exp\bigg(-t-\frac{x^2}{4\beta^2t}\bigg)I_{\alpha_1+\alpha_2-1}\bigg(\frac{|x|}{\beta}\sqrt{\frac{2\lambda}{t}}\bigg)\,\mathrm{d}t.
	\end{align}
\end{theorem}

\section{Asymptotic expansions relating to the distribution}\label{sec3}

We begin by studying the asymptotic behaviour of the PDFs of the random variables $X_1-X_2$ and $X_1+X_2$ at the origin. One of the formulas given in the following proposition is stated in terms of the Horn function $\Psi_2$, given by the double infinite series
\begin{align}\label{horn}
\Psi_2(a;b_1,b_2;x,y)=\sum_{m=0}^\infty \sum_{n=0}^\infty\frac{(a)_{m+n}}{(b_1)_m(b_2)_n}\frac{x^m}{m!}\frac{y^n}{n!},   
\end{align}
which is absolutely convergent for all $x,y\in\mathbb{R}$ (see \cite[p.\ 225]{e53}).

\begin{proposition}\label{prop2.4}
	1. Suppose that $\alpha_1+\alpha_2<1$. Then, as $x\rightarrow0$,
	\begin{equation}\label{diff0lim1}
		f_{X_1-X_2}(x)\sim \frac{1}{\pi}\beta_1^{-\alpha_1}\beta_2^{-\alpha_2}\mathrm{e}^{-\lambda_1-\lambda_2}\sin(\pi a_{0,0}(x))\Gamma(1-\alpha_1-\alpha_2)\,|x|^{\alpha_1+\alpha_2-1},
	\end{equation}
    where $a_{0,0}(x)=\alpha_1$ for $x>0$ and $a_{0,0}(x)=\alpha_2$ for $x<0$.

    \vspace{3mm}
    
	\noindent 2. Suppose that $\alpha_1+\alpha_2=1$. Then, as $x\rightarrow0$,
	\begin{equation}\label{diff0lim2}
	f_{X_1-X_2}(x)\sim-\frac{1}{\pi}\beta_1^{-\alpha_1}\beta_2^{-\alpha_2}\mathrm{e}^{-\lambda_1-\lambda_2}\sin(\pi a_{0,0}(x))\ln|x|.
	\end{equation}
    3. The distribution of $X_1-X_2$ is unimodal with mode 0 if $\alpha_1+\alpha_2\leq 1$.

    \vspace{3mm}

    \noindent
4. Suppose that $\alpha_1+\alpha_2>1$. Then
\begin{align}
f_{X_1-X_2}(0)&=\mathrm{e}^{-\lambda_1-\lambda_2} \beta_1^{-\alpha_1}\beta_2^{-\alpha_2}\bigg(\frac{1}{\beta_1}+\frac{1}{\beta_2}\bigg)^{1-\alpha_1-\alpha_2} \frac{\Gamma(\alpha_1+\alpha_2-1)}{\Gamma(\alpha_1)\Gamma(\alpha_2)}\nonumber\\
&\quad\times  \Psi_2\bigg(\alpha_1+\alpha_2-1,\alpha_1,\alpha_2,\frac{\lambda_1\beta_2}{\beta_1+\beta_2},\frac{\lambda_2\beta_1}{\beta_1+\beta_2}\bigg). \label{hornf}
\end{align}

    \vspace{3mm}

    \noindent
	5. The PDF $f_{X_1-X_2}(x)$ is bounded for all $x\in\mathbb{R}$ if and only if $\alpha_1+\alpha_2>1$.
\end{proposition}

\begin{proposition}\label{prop2.5}
1. Let $\alpha_1,\alpha_2>0$. Then, as $x\rightarrow0$,
	\begin{equation}\label{lim7}
		f_{X_1+X_2}(x)\sim \frac{\beta_1^{-\alpha_1}\beta_2^{-\alpha_2}}{\Gamma(\alpha_1+\alpha_2)}\mathrm{e}^{-\lambda_1-\lambda_2}x^{\alpha_1+\alpha_2-1}.
	\end{equation}
  2. The distribution of $X_1+X_2$ is unimodal with mode 0 if $\alpha_1+\alpha_2< 1$.

    \vspace{3mm}

    \noindent
	3. The PDF $f_{X_1+X_2}(x)$ is bounded for all $x\in\mathbb{R}$ if and only if $\alpha_1+\alpha_2>1$.    
\end{proposition}

\begin{remark} Since $\Psi_2(a;b_1,b_2;0,y)=M(a,b_2,y)$ and $\Psi_2(a;b_1,b_2;x,0)=M(a,b_1,x)$ (which is immediate from comparison of the power series representation (\ref{horn}) of the Horn function $\Psi_2$ and the power series representation of $M(a,b,x)$ given in \cite[Eq.\ 13.2.2]{olver}), formula (\ref{hornf}) can be expressed in terms of the confluent hypergeometric function of the first kind when $\lambda_1=0$ or $\lambda_2=0$.
When $\lambda_1=0$ we have that
\begin{align*}
f_{X_1-X_2}(0)&=\mathrm{e}^{-\lambda_2} \beta_1^{-\alpha_1}\beta_2^{-\alpha_2}\bigg(\frac{1}{\beta_1}+\frac{1}{\beta_2}\bigg)^{1-\alpha_1-\alpha_2} \frac{\Gamma(\alpha_1+\alpha_2-1)}{\Gamma(\alpha_1)\Gamma(\alpha_2)}\\
&\quad\times M\bigg(\alpha_1+\alpha_2-1,\alpha_2,\frac{\lambda_2\beta_1}{\beta_1+\beta_2}\bigg),  
\end{align*}
and we get a similar simplification when $\lambda_2=0$. Also, since $M(a,b,0)=1$, we have, for $\lambda_1=\lambda_2=0$,
\begin{align*}
f_{X_1-X_2}(0)=\beta_1^{-\alpha_1}\beta_2^{-\alpha_2}\bigg(\frac{1}{\beta_1}+\frac{1}{\beta_2}\bigg)^{1-\alpha_1-\alpha_2} \frac{\Gamma(\alpha_1+\alpha_2-1)}{\Gamma(\alpha_1)\Gamma(\alpha_2)},   
\end{align*}
which is in agreement with the expression given in the displayed equation below (2.1) of \cite{h21}.
\end{remark}

In the following theorem, we provide asymptotic expansions for the PDFs $f_{X_1-X_2}(x)$ and $f_{X_1+X_2}(x)$ as $x\rightarrow\pm\infty$. For the limit $x\rightarrow\infty$, we are able to state unified asymptotic expansions that hold for both the difference $X_1-X_2$ and the sum $X_1+X_2$ using the $\mp$ and $\pm$ notation: that is we provide asymptotic expansions for $f_{X_1\mp X_2}(x)$ as $x\rightarrow\infty$. 

\begin{theorem}\label{thm3.1} Let $\alpha_1,\alpha_2,\beta_1,\beta_2>0$ and $\lambda_1,\lambda_2\geq0$. For the asymptotic expansions involving the sum $X_1+X_2$ we further suppose that $\beta_1>\beta_2$.

\vspace{3mm}

\noindent 1. Suppose that $\lambda_1\neq0$.
	Then, as $x\rightarrow\infty$,
	\begin{align}\label{thm3.11}
		f_{X_1\mp X_2}(x)&\sim\frac{1}{2\sqrt{\pi}}\bigg(\frac{1}{\beta_2}\pm\frac{1}{\beta_1}\bigg)^{-\alpha_2}\beta_1^{-\alpha_1}\beta_2^{-\alpha_2}\bigg(\frac{\lambda_1}{\beta_1}\bigg)^{(1-2\alpha_1)/4}\exp\bigg(-\lambda_1-\frac{\lambda_2\beta_2}{\beta_2\pm\beta_1}\bigg)\nonumber\\
		&\quad\times x^{(2\alpha_1-3)/4}\exp\bigg(2\sqrt{\frac{\lambda_1x}{\beta_1}}-\frac{x}{\beta_1}\bigg)\sum_{l=0}^{\infty}\frac{c_l(\alpha_1,\alpha_2,\beta_1,\beta_2,\lambda_1,\lambda_2)}{x^{l/2}},
	\end{align}
	where $c_0(\alpha_1,\alpha_2,\beta_1,\beta_2,\lambda_1,\lambda_2)=1$ and
\begin{align}
c_l(\alpha_1,\alpha_2,\beta_1,\beta_2,\lambda_1,\lambda_2)&=
\sum_{j=0}^l\frac{(3/2+j-\alpha_1)_{l-j}(\alpha_1-j-1/2)_{l-j}}{(l-j)!4^{l-j}}\nonumber\\
&\quad\times \bigg(\frac{\beta_1}{\lambda_1}\bigg)^{l/2-j}\bigg(\frac{1}{\beta_1}\pm\frac{1}{\beta_2}\bigg)^{-j} L_j^{(\alpha_2-1)}\bigg(-\frac{\lambda_2\beta_1}{\beta_1\pm\beta_2}\bigg), \quad l\geq1. \label{c1for} 
\end{align}    
	2. Suppose that $\lambda_1=0$. Then, as $x\rightarrow\infty$,
	\begin{align}
		f_{X_1\mp X_2}(x)\sim \frac{\beta_1^{-\alpha_1}\beta_2^{-\alpha_2}}{\Gamma(\alpha_1)}\bigg(\frac{1}{\beta_2}\pm\frac{1}{\beta_1}\bigg)^{-\alpha_2}x^{\alpha_1-1}\exp\bigg(-\frac{x}{\beta_1}-\frac{\lambda_2\beta_2}{\beta_2\pm\beta_1}\bigg)\sum_{k=0}^{\infty}\frac{d_k(\alpha_1,\alpha_2,\beta_1,\beta_2,\lambda_2)}{x^k},\label{thm3.12}
	\end{align}
		where $d_0(\alpha_1,\alpha_2,\beta_1,\beta_2,\lambda_2)=1$ and
	\begin{align}
		d_k(\alpha_1,\alpha_2,\beta_1,\beta_2,\lambda_2)&=(-1)^k(1-\alpha_1)_k\bigg(\frac{1}{\beta_1}\pm\frac{1}{\beta_2}\bigg)^{-k}L_k^{(\alpha_2-1)}\bigg(-\frac{\lambda_2\beta_1}{\beta_1\pm\beta_2}\bigg), \quad k\geq1. \nonumber
	\end{align}
3. Suppose that $\lambda_2\neq0$.
	Then, as $x\rightarrow-\infty$,
	\begin{align}\label{thm3.11aa}
		f_{X_1- X_2}(x)&\sim\frac{1}{2\sqrt{\pi}}\bigg(\frac{1}{\beta_1}+\frac{1}{\beta_2}\bigg)^{-\alpha_1}\beta_1^{-\alpha_1}\beta_2^{-\alpha_2}\bigg(\frac{\lambda_2}{\beta_2}\bigg)^{(1-2\alpha_2)/4}\exp\bigg(-\lambda_2-\frac{\lambda_1\beta_1}{\beta_1+\beta_2}\bigg)\nonumber\\
		&\quad\times |x|^{(2\alpha_2-3)/4}\exp\bigg(2\sqrt{\frac{\lambda_2|x|}{\beta_2}}+\frac{x}{\beta_2}\bigg)\sum_{l=0}^{\infty}\frac{c_l(\alpha_2,\alpha_1,\beta_2,\beta_1,\lambda_2,\lambda_1)}{x^{l/2}}.
	\end{align}
	4. Suppose that $\lambda_2=0$. Then, as $x\rightarrow-\infty$,
	\begin{align}\label{thm3.14}
		f_{X_1-X_2}(x)\sim\frac{\beta_1^{-\alpha_1}\beta_2^{-\alpha_2}}{\Gamma(\alpha_2)}\bigg(\frac{1}{\beta_1}+\frac{1}{\beta_2}\bigg)^{-\alpha_1}|x|^{\alpha_2-1}\exp\bigg(\frac{x}{\beta_2}-\frac{\lambda_1\beta_1}{\beta_1+\beta_2}\bigg)\sum_{k=0}^{\infty}\frac{d_k(\alpha_2,\alpha_1,\beta_2,\beta_1,\lambda_1)}{|x|^k}.
	\end{align}
\end{theorem}

\begin{remark}\label{remdiff}
The asymptotic expansions for the PDF $f_{X_1+X_2}(x)$ of the sum $X_1+X_2$ are made under the assumption $\beta_1>\beta_2$. This assumption is non-restrictive because for the sum $X_1+X_2$ if we are working in the case that $\beta_1\not=\beta_2$ then we may assume without loss of generality that $\beta_1>\beta_2$. We could obtain analogues of our results from Theorem \ref{thm3.1} (and the forthcoming Theorems \ref{thm3.5} and \ref{thm3.7}) for the sum $X_1+X_2$ in the case $\beta_1=\beta_2$ but elect not to do this, because in this case $X_1+X_2$ follows the non-central gamma distribution (see part 1 of Remark \ref{zero}), and asymptotic approximations relating to the non-central gamma distribution have already received attention in the literature; see, for example, \cite{t93}.
\end{remark}

\begin{remark}\label{remexp} When $\alpha_1\in\mathbb{Z}^+$ the asymptotic expansion (\ref{thm3.12}) for the PDF $f_{X_1-X_2}(x)$ is exact, whilst when $\alpha_2\in\mathbb{Z}^+$ the asymptotic expansion (\ref{thm3.14}) for the PDF $f_{X_1-X_2}(x)$ is exact. This follows from comparison to the exact finite series representations (\ref{rprp}) and (\ref{rprp2}) for the PDF $f_{X_1-X_2}(x)$ that are given in Proposition \ref{prop1}. On the other hand, when $\alpha_1\in\mathbb{Z}^+$ the asymptotic expansion (\ref{thm3.12}) for the PDF $f_{X_1+X_2}(x)$ is only exact up to an exponentially small remainder term, as is the case for the asymptotic expansion (\ref{thm3.14}) for the PDF $f_{X_1+X_2}(x)$ when $\alpha_2\in\mathbb{Z}^+$. That in the case of the sum $X_1+X_2$ the asymptotic expansions (\ref{thm3.12}) and (\ref{thm3.14}) are correct up to an exponentially small remainder term when $\alpha_1\in\mathbb{Z}^+$ and $\alpha_2\in\mathbb{Z}^+$, respectively, can be seen by examining the proof of Theorem \ref{thm3.1}, and the numerical results in Table \ref{table4} are consistent with the fact that the asymptotic expansion (\ref{thm3.12}) is correct up to an exponentially small remainder in the case $\alpha_1=1$ (since non-zero but rapidly decreasing relative errors are observed as $x$ increases).
\end{remark}

\begin{remark}   When $\lambda_2=0$ the coefficients $c_l(\alpha_1,\alpha_2,\beta_1,\beta_2,\lambda_1,0)$, $l\geq1$, simplify to
\begin{align*}
c_l(\alpha_1,\alpha_2,\beta_1,\beta_2,\lambda_1,0)
&=\sum_{j=0}^l\frac{(3/2+j-\alpha_1)_{l-j}(\alpha_1-j-1/2)_{l-j}(\alpha_2)_j}{j!(l-j)!4^{l-j}}\bigg(\frac{\beta_1}{\lambda_1}\bigg)^{l/2-j}\bigg(\frac{1}{\beta_1}\pm\frac{1}{\beta_2}\bigg)^{-j},
\end{align*}
since $L_n^{(\alpha)}(0)=(\alpha+1)_n/n!$ (see \cite[Eq.\ 18.6.1]{olver}). 
\end{remark}

\begin{remark}\label{remcd} Using the special values $g_{0,0}(a,b)=1$, $g_{1,1}(a,b)=b/2$, $g_{1,2}(a,b)=a$ and $g_{2,2}(a,b)=b^2/8$ and the explicit formulas
\begin{align*}
L_0^{(\alpha)}(x)=1, \quad L_1^{(\alpha)}(x)=\alpha+1-x, \quad L_2^{(\alpha)}(x)=\frac{1}{2}(\alpha+1)(\alpha+2)-(\alpha+2)x+\frac{1}{2}x^2    
\end{align*}
(see \cite[Eq.\ 18.5.17\textunderscore5]{olver}) we obtain the following explicit formulas for the coefficients $c_l=c_l(\alpha_1,\alpha_2,\beta_1,\beta_2,\lambda_1,\lambda_2)$, $l=1,2$:
	\begin{align}
		c_1&=\frac{(3/2-\alpha_1)(\alpha_1-1/2)}{4}\sqrt{\frac{\beta_1}{\lambda_1}}+\frac{\beta_1\beta_2}{\beta_2\pm\beta_1}\sqrt{\frac{\lambda_1}{\beta_1}}\bigg(\alpha_2+\frac{\lambda_2\beta_1}{\beta_1\pm\beta_2}\bigg), \nonumber \\
		c_2&=\frac{(\alpha_1-3/2)(\alpha_1-5/2)(\alpha_1-1/2)(\alpha_1+1/2)}{32}\frac{\beta_1}{\lambda_1}\nonumber\\
        &\quad-\frac{(\alpha_1-3/2)(\alpha_1-5/2)}{4}\frac{\beta_1\beta_2}{\beta_2\pm\beta_1}\bigg(\alpha_2+\frac{\lambda_2\beta_1}{\beta_1\pm\beta_2}\bigg) \nonumber \\
		&\quad+\frac{\lambda_1}{2\beta_1}\bigg(\frac{\beta_1\beta_2}{\beta_1\pm\beta_2}\bigg)^2\bigg(\alpha_2(\alpha_2+1)+\frac{2(\alpha_2+1)\lambda_2\beta_1}{\beta_1\pm\beta_2}+\frac{\lambda_2^2\beta_1^2}{(\beta_1\pm\beta_2)^2}\bigg). \nonumber
	\end{align}
	Similarly, we obtain the following expressions for the coefficients $d_k=d_k(\alpha_1,\alpha_2,\beta_1,\beta_2,\lambda_2)$, $k=1,2$:
	\begin{align*}
		d_1&=(\alpha_1-1)\frac{\beta_1\beta_2}{\beta_2\pm\beta_1}\bigg(\alpha_2+\frac{\lambda_2\beta_1}{\beta_1\pm\beta_2}\bigg),\\
		d_2&=\frac{(\alpha_1-1)(\alpha_1-2)}{2}\bigg(\frac{\beta_1\beta_2}{\beta_1\pm\beta_2}\bigg)^2\bigg(\alpha_2(\alpha_2+1)+\frac{2(\alpha_2+1)\lambda_2\beta_1}{\beta_1\pm\beta_2}+\frac{\lambda_2^2\beta_1^2}{(\beta_1\pm\beta_2)^2}\bigg).
	\end{align*}
\end{remark}

We now state asymptotic expansions for the tail probabilities of the difference $X_1-X_2$ and the sum $X_1+X_2$. We write $F_{X_1\mp X_2}(x)=\mathbb{P}(X_1\mp X_2\leq x)$ and $\overline{F}_{X_1\mp X_2}(x)=1-F_{X_1\mp X_2}(x)=\mathbb{P}(X_1\mp X_2>x)$.

\begin{theorem}\label{thm3.5} Let $\alpha_1,\alpha_2,\beta_1,\beta_2>0$ and $\lambda_1,\lambda_2\geq0$. For the asymptotic expansions involving the sum $X_1+X_2$ we further suppose that $\beta_1>\beta_2$.

\vspace{3mm}

\noindent 1. Suppose that $\lambda_1\neq0$. Then, as $x\rightarrow\infty$,
	\begin{align}
	\overline{F}_{X_1\mp X_2}(x)&\sim\frac{1}{2\sqrt{\pi}}\bigg(\frac{1}{\beta_2}\pm\frac{1}{\beta_1}\bigg)^{-\alpha_2}\beta_1^{1-\alpha_1}\beta_2^{-\alpha_2}\bigg(\frac{\lambda_1}{\beta_1}\bigg)^{(1-2\alpha_1)/4}\exp\bigg(-\lambda_1-\frac{\lambda_2\beta_2}{\beta_2\pm\beta_1}\bigg)\nonumber\\
	&\quad\times x^{(2\alpha_1-3)/4}\exp\bigg(2\sqrt{\frac{\lambda_1 x}{\beta_1}}-\frac{x}{\beta_1}\bigg)\sum_{p=0}^{\infty}\frac{\gamma_p(\alpha_1,\alpha_2,\beta_1,\beta_2,\lambda_1,\lambda_2)}{x^{p/2}},\label{thm3.61}
	\end{align}
	where $\gamma_0(\alpha_1,\alpha_2,\beta_1,\beta_2,\lambda_1,\lambda_2)=1$ and
	\begin{align}
		\gamma_p(\alpha_1,\alpha_2,\beta_1,\beta_2,\lambda_1,\lambda_2)&=\sum_{\substack{i,j,k,l\geq0 \\ i+2j+k+l=p}}(-1)^{i+j}c_l(\alpha_1,\alpha_2,\beta_1,\beta_2,\lambda_1,\lambda_2)\binom{\alpha_1-1/2-l}{k}\nonumber\\
		&\quad\times\binom{\alpha_1-3/2-l-k-2j}{i}\big((l+k)/2-(2\alpha_1-3)/4\big)_j\nonumber\\
        &\quad\times\beta_1^{j}\big(\lambda_1\beta_1\big)^{(k+i)/2}, \quad p\geq1,\nonumber
	\end{align}
	with $c_l(\alpha_1,\alpha_2,\beta_1,\beta_2,\lambda_1,\lambda_2)$, $l\geq0$, defined as in part 1 of Theorem \ref{thm3.1}.

\vspace{3mm}
    
\noindent	2. Suppose that $\lambda_1=0$. Then, as $x\rightarrow\infty$,
\begin{align}\label{thm3.52}
	\overline{F}_{X_1\mp X_2}(x)&\sim\frac{\beta_1^{1-\alpha_1}\beta_2^{-\alpha_2}}{\Gamma(\alpha_1)}\bigg(\frac{1}{\beta_2}\pm\frac{1}{\beta_1}\bigg)^{-\alpha_2}{x^{\alpha_1-1}}\exp\bigg(-\frac{x}{\beta_1}-\frac{\lambda_2\beta_2}{\beta_2\pm\beta_1}\bigg)\sum_{k=0}^{\infty}\frac{\delta_k(\alpha_1,\alpha_2,\beta_1,\beta_2,\lambda_2)}{x^k},
	\end{align}
	where $\delta_0(\alpha_1,\alpha_2,\beta_1,\beta_2,\lambda_2)=1$ and
	\begin{equation}
	\delta_k(\alpha_1,\alpha_2,\beta_1,\beta_2,\lambda_2)=\sum_{s=0}^{k}(-\beta_1)^{k-s}(k+1-\alpha_1)_{k-s}\,d_s(\alpha_1,\alpha_2,\beta_1,\beta_2,\lambda_2), \quad k\geq1,\nonumber
	\end{equation}
	with $d_s(\alpha_1,\alpha_2,\beta_1,\beta_2,\lambda_2)$, $s\geq0$, defined as in part 2 of Theorem \ref{thm3.1}.

        \vspace{3mm}
    
\noindent	3. Suppose that $\lambda_2\not=0$. Then, as $x\rightarrow-\infty$,
	\begin{align}
	F_{X_1-X_2}(x)&\sim\frac{1}{2\sqrt{\pi}}\bigg(\frac{1}{\beta_1}+\frac{1}{\beta_2}\bigg)^{-\alpha_1}\beta_1^{-\alpha_1}\beta_2^{1-\alpha_2}\bigg(\frac{\lambda_2}{\beta_2}\bigg)^{(1-2\alpha_2)/4}\exp\bigg(-\lambda_2-\frac{\lambda_1\beta_1}{\beta_1+\beta_2}\bigg)\nonumber\\
	&\quad\times |x|^{(2\alpha_2-3)/4}\exp\bigg(2\sqrt{\frac{\lambda_2 |x|}{\beta_2}}+\frac{x}{\beta_2}\bigg)\sum_{p=0}^{\infty}\frac{\gamma_p(\alpha_2,\alpha_1,\beta_2,\beta_1,\lambda_2,\lambda_1)}{|x|^{p/2}}.\label{negtail}
	\end{align}
4. Suppose that $\lambda_2=0$. Then, as $x\rightarrow-\infty$,
	\begin{align}\label{thm3.54}
		F_{X_1-X_2}(x)&\sim\frac{\beta_1^{-\alpha_1}\beta_2^{1-\alpha_2}}{\Gamma(\alpha_2)}\bigg(\frac{1}{\beta_1}+\frac{1}{\beta_2}\bigg)^{-\alpha_1}|x|^{\alpha_2-1}\exp\bigg(\frac{x}{\beta_2}-\frac{\lambda_1\beta_1}{\beta_1+\beta_2}\bigg)\sum_{k=0}^{\infty}\frac{\delta_k(\alpha_2,\alpha_1,\beta_2,\beta_1,\lambda_1)}{|x|^k}.
	\end{align}
\end{theorem}

\begin{remark} The coefficients $\gamma_k=\gamma_k(\alpha_1,\alpha_2,\beta_1,\beta_2,\lambda_1,\lambda_2)$ and $\delta_k=\delta_k(\alpha_1,\alpha_2,\beta_1,\beta_2,\lambda_2)$, $k=1,2$ can be expressed explicitly as follows:
	\begin{align*}
		\gamma_1&=c_1+\sqrt{\lambda_1\beta_1},\\
\gamma_2&=c_2+c_1\sqrt{\lambda_1\beta_1}+\beta_1(2\alpha_1-3)/4+\lambda_1\beta_1,
	\end{align*}
where expressions for $c_1$ and $c_2$ are given in Remark \ref{remcd}, and 
\begin{align*}
		\delta_1&=(\alpha_1-2)\beta_1+(\alpha_1-1)\frac{\beta_1\beta_2}{\beta_2\pm\beta_1}\bigg(\alpha_2+\frac{\lambda_2\beta_1}{\beta_1\pm\beta_2}\bigg),\\
	\delta_2&=(\alpha_1-3)(\alpha_1-4)\beta_1^2+(\alpha_1-1)(\alpha_1-3)\frac{\beta_1^2\beta_2}{\beta_2\pm\beta_1}\bigg(\alpha_2+\frac{\lambda_2\beta_1}{\beta_1\pm\beta_2}\bigg)\\
	&\quad+\frac{(\alpha_1-1)(\alpha_1-2)\beta_1^2\beta_2^2}{2(\beta_1\pm\beta_2)^2}\bigg(\alpha_2(\alpha_2+1)+\frac{2(\alpha_2+1)\lambda_2\beta_1}{\beta_1\pm\beta_2}+\frac{\lambda_2^2\beta_1^2}{(\beta_1\pm\beta_2)^2}\bigg).
\end{align*}
\end{remark}

For $0<p<1$, denote the quantile function of $X_1\mp X_2$ by $Q_{X_1\mp X_2}(p)=F^{-1}_{X_1\mp X_2}(p)$. In the following theorem, we present asymptotic approximations for the quantile function.

\begin{theorem}\label{thm3.7}Let $\alpha_1,\alpha_2,\beta_1,\beta_2>0$, $\lambda_1,\lambda_2\geq0$ and also let $q=1-p$. For the asymptotic approximations involving the sum $X_1+X_2$ we further suppose that $\beta_1>\beta_2$. 

\vspace{3mm}

\noindent	1. Suppose that $\lambda_1\neq0$. Then, as $p\rightarrow1$,
	\begin{align}\label{thm3.71}
	Q_{X_1\mp X_2}(p)&=\beta_1\bigg(\ln(1/q)+2\sqrt{\lambda_1}\sqrt{\ln(1/q)}+\frac{2\alpha_1-3}{4}\ln(\ln(1/q))\nonumber\\
	&\quad+\ln\bigg(\frac{1}{2\sqrt{\pi}}\bigg(\frac{1}{\beta_2}\pm\frac{1}{\beta_1}\bigg)^{-\alpha_2}\beta_2^{-\alpha_2}\lambda_1^{(1-2\alpha_1)/4}\bigg)+\lambda_1-\frac{\lambda_2\beta_2}{\beta_2\pm\beta_1}\nonumber\\
	&\quad+\frac{(2\alpha_1-3)\sqrt{\lambda_1}}{4}\frac{\ln(\ln(1/q))}{\sqrt{\ln(1/q)}}\bigg)+O\bigg(\frac{1}{\sqrt{\ln(1/q)}}\bigg).
	\end{align}
	2. Suppose that $\lambda_1=0$. Then, as $p\rightarrow1$,
	\begin{align}\label{thm3.72}
		Q_{X_1\mp X_2}(p)&=\beta_1\bigg(\ln(1/q)+(\alpha_1-1)\ln(\ln(1/q))+\ln\bigg(\frac{\big(1\pm\beta_2/\beta_1\big)^{-\alpha_2}}{\Gamma(\alpha_1)}\bigg)-\frac{\lambda_2\beta_2}{\beta_2\pm\beta_1}\bigg)\nonumber\\
		&\quad+O\bigg(\frac{1}{{\ln(1/q)}}\bigg).
	\end{align}
3. Suppose that $\lambda_2\not=0$. Then, as $p\rightarrow0$,
\begin{align}
	Q_{X_1- X_2}(p)&=-\beta_2\bigg(\ln(1/p)+2\sqrt{\lambda_2}\sqrt{\ln(1/p)}+\frac{2\alpha_2-3}{4}\ln(\ln(1/p))\nonumber\\
	&\quad+\ln\bigg(\frac{1}{2\sqrt{\pi}}\bigg(\frac{1}{\beta_1}+\frac{1}{\beta_2}\bigg)^{-\alpha_1}\beta_1^{-\alpha_1}\lambda_2^{(1-2\alpha_2)/4}\bigg)-\frac{\lambda_1\beta_1}{\beta_1+\beta_2}\nonumber\\
	&\quad+\frac{(2\alpha_2-3)\sqrt{\lambda_2}}{4}\frac{\ln(\ln(1/p))}{\sqrt{\ln(1/p)}}\bigg)+O\bigg(\frac{1}{\sqrt{\ln(1/p)}}\bigg). \label{qq00}
	\end{align}
	4. Suppose that $\lambda_2=0$. Then, as $p\rightarrow0$,
	\begin{align}\label{thm3.74}
		Q_{X_1-X_2}(p)&=-\beta_2\bigg(\ln(1/p)+(\alpha_2-1)\ln(\ln(1/p))+\ln\bigg(\frac{\big({\beta_1}/\beta_2+1\big)^{-\alpha_1}}{\Gamma(\alpha_2)}\bigg)-\frac{\lambda_1\beta_1}{\beta_1+\beta_2}\bigg)\nonumber\\
		&\quad+O\bigg(\frac{1}{{\ln(1/p)}}\bigg).
	\end{align}
\end{theorem}

\begin{remark}
In the case $\lambda_1\not=0$, the asymptotic approximation (\ref{thm3.71}) includes the first five terms in the asymptotic expansion of the quantile function $Q_{X_1\mp X_2}(p)$ as $p\rightarrow1$, whilst in the case $\lambda_1=0$ the asymptotic approximation (\ref{thm3.72}) contains the first three terms in the asymptotic expansion. These asymptotic approximations were derived using just the leading-order term in the asymptotic expansions (\ref{thm3.61}) and (\ref{thm3.52}) for the tail probabilities $P(X_1\mp X_2>x)$, as given in Theorem \ref{thm3.5}. Further terms in the asymptotic expansions of the quantile function $Q_{X_1\mp X_2}(p)$ as $p\rightarrow1$ could be obtained by including further terms in the asymptotic expansions (\ref{thm3.61}) and (\ref{thm3.52}) for the tail probabilities in the derivation. Similar comments apply to the asymptotic approximations (\ref{qq00}) and (\ref{thm3.74}) for the quantile function $Q_{X_1-X_2}(p)$ as $p\rightarrow0$. 
\end{remark}

\section{Special case: the product of correlated normal random variables}

Let $(X,Y)$ follow the bivariate normal distribution with mean vector $(\mu_X,\mu_Y)\in\mathbb{R}^2$, variances $\sigma_X^2,\sigma_Y^2>0$ and correlation coefficient $\rho\in(-1,1)$. Denote the product by $Z=XY$, and also let $S_n=\sum_{i=1}^nZ_i$, where $Z_1,\ldots,Z_n$ are independent copies of $Z$. Also, for sake of simpler expressions, set $r_X=\mu_X/\sigma_X$ and $r_Y=\mu_Y/\sigma_Y$. Asymptotic expansions for the PDF of the sum $S_n$ as $x\rightarrow\pm\infty$ were obtained by \cite[Theorem 3.1]{gz25}, although, for general parameter values, the coefficients were not given in a fully explicit form, which made implementation of the asymptotic expansions of the PDF, and in turn the asymptotic expansions for the tail probabilities, more challenging. Previously, the leading-order asymptotic behaviour of the PDF of the product $Z$ as $x\rightarrow\pm\infty$ was established by \cite{gz23}, whilst the leading-order asymptotic behaviour of the PDF and tail probabilities of the product of $n$ independent normal random variables has been studied by \cite{cs26,g17,l23}.


We now deduce asymptotic expansions for the PDF of the sum $S_n$ from the asymptotic expansions for the PDF of the non-central gamma distribution given in Theorem \ref{thm3.1}. We begin by recalling the following representation of $S_n$ in terms of the difference of two independent scaled non-central chi-square random variables. Theorem 2.1 of \cite{g25} states that
\begin{align*}
S_n=_d\frac{s}{2}(1+\rho)V_1-\frac{s}{2}(1-\rho)V_2,    
\end{align*}
where $V_1\sim \chi_n'^2(\lambda_+)$ and $V_2\sim\chi_n'^2(\lambda_-)$ are independent non-central chi-square random variables with non-centrality parameters
\begin{align*}
\lambda_+=\frac{n(r_X+ r_Y)^2}{2(1+ \rho)}, \quad  \lambda_-=\frac{n(r_X- r_Y)^2}{2(1- \rho)}.  
\end{align*}
Applying the distributional relation (\ref{chirep}) we therefore have that
\begin{align}\label{snsnsn}
S_n\sim\mathrm{NCGD}\bigg(\frac{n}{2},\frac{n}{2},s(1+\rho),s(1-\rho),\frac{n(r_X+r_Y)^2}{4(1+\rho)},\frac{n(r_X-r_Y)^2}{4(1-\rho)}\bigg).    
\end{align}
From the distributional relation (\ref{snsnsn}) we deduce the following asymptotic expansions for the PDF of $S_n$ directly form the asymptotic expansions of Theorem \ref{thm3.1}.

\begin{corollary}\label{corpn} 
1. Suppose that $r_X+r_Y\not=0$. Then, as $x\rightarrow\infty$,
\begin{align}f_{S_n}(x)&\sim C(r_X,r_Y,\rho,n)\,x^{(n-3)/4}\exp\bigg(\frac{|r_X+r_Y|}{1+\rho}\sqrt{\frac{nx}{s}}-\frac{x}{s(1+\rho)}\bigg)\sum_{l=0}^\infty c_l(r_X,r_Y,\rho,n)\frac{s^{l/2}}{x^{l/2}}, \label{expansion11}
\end{align}
where 
\begin{align*}
C(r_X,r_Y,\rho,n)&= \frac{s^{-(n+1)/4}}{2\sqrt{2\pi}}\bigg(\frac{1+\rho}{|r_X+r_Y|\sqrt{n}}\bigg)^{(n-1)/2}\\
&\quad\times\exp\bigg(-\frac{n}{4(1+\rho)}(r_X+r_Y)^2\bigg)\exp\bigg(-\frac{n}{8}(r_X-r_Y)^2\bigg),   
\end{align*}
and $c_0(r_X,r_Y,\rho,n)=1$ and, for $l\geq1$,
\begin{align*}
c_l(r_X,r_Y,\rho,n)&=\frac{1}{2^ln^{l/2}}\bigg(\frac{1+\rho}{|r_X+r_Y|}\bigg)^l\sum_{j=0}^l\frac{(3/2-n/2+j)_{l-j}(n/2-1/2-j)_{l-j}}{(l-j)!}\nonumber\\
&\quad\times \bigg(\frac{n}{2}\bigg(\frac{1-\rho}{1+\rho}\bigg)(r_X+r_Y)^2\bigg)^j L_j^{(n/2-1)}\bigg(-\frac{n}{8}\bigg(\frac{1+\rho}{1-\rho}\bigg)(r_X-r_Y)^2\bigg).   
\end{align*}

\vspace{3mm}

\noindent 2. Suppose that $r_X+r_Y=0$. Then, as $x\rightarrow\infty$,  
\begin{align}f_{S_n}(x)\sim \frac{\mathrm{e}^{-nr_X^2/2}}{(2s)^{n/2}\Gamma(n/2)} x^{n/2-1}\exp\bigg\{-\frac{x}{s(1+\rho)}\bigg\}\sum_{k=0}^\infty d_k(r_X,\rho,n)\frac{s^k}{x^k}, \label{expan11}
\end{align}  
where $d_0(r_X,\rho,n)=1$ and, for $k\geq1$,
\begin{align*}
d_k(r_X,\rho,n)=(-1)^k\bigg(1-\frac{n}{2}\bigg)_k\bigg(\frac{1-\rho^2}{2}\bigg)^kL_k^{(n/2-1)}\bigg(-\frac{n}{2}\bigg(\frac{1+\rho}{1-\rho}\bigg)r_X^2\bigg).    
\end{align*}

\vspace{3mm}

\noindent 3. Suppose that $r_X-r_Y\not=0$. Then, an asymptotic expansion for $f_{S_n}(x)$ as $x\rightarrow-\infty$ is obtained by replacing $(r_Y,\rho,x)$ by $(-r_Y,-\rho,-x)$ in the asymptotic expansion (\ref{expansion11}).

\vspace{3mm}

\noindent 4. Suppose that $r_X-r_Y=0$. Then, an asymptotic expansion for $f_{S_n}(x)$ as $x\rightarrow-\infty$ is obtained by replacing $(\rho,x)$ by $(-\rho,-x)$ in the asymptotic expansion (\ref{expan11}).
\end{corollary} 

Setting $n=1$ in the formulas of Corollary \ref{corpn} yields asymptotic expansions for the PDF of the product $Z$. It turns out that in this case we can obtain further simplifications, which are given in the following corollary; the short proof is given in Section \ref{sec5}. The coefficients are expressed in terms of the ``physicist's  Hermite polynomials," which are given by
\begin{align*}
H_n(x)=(-1)^n\mathrm{e}^{x^2}\frac{\mathrm{d}^n}{\mathrm{d}x^n}\big(\mathrm{e}^{-x^2}\big)   
\end{align*}
(see \cite[Chapter 18]{olver}). In the case $r_X+r_Y\not=0$, the coefficients involve the ceiling function $\lceil x\rceil=\min\{k\in\mathbb{Z}\,:\,k\geq x\}$.

\begin{corollary}\label{corpn2} 
1. Suppose that $r_X+r_Y\not=0$. Then, as $x\rightarrow\infty$,
\begin{align}
f_{Z}(x)&\sim \frac{1}{2\sqrt{2\pi sx}}\exp\bigg(-\frac{1}{4(1+\rho)}(r_X+r_Y)^2\bigg)\exp\bigg(-\frac{1}{8}(r_X-r_Y)^2\bigg)\nonumber \\
&\quad\times \exp\bigg(\frac{|r_X+r_Y|}{1+\rho}\sqrt{\frac{x}{s}}-\frac{x}{s(1+\rho)}\bigg)\sum_{l=0}^\infty c_l(r_X,r_Y,\rho,1)\frac{s^{l/2}}{x^{l/2}},\label{33a}
\end{align}
where $c_0(r_X,r_Y,\rho,1)=1$ and, for $l\geq1$,
\begin{align}
c_l(r_X,r_Y,\rho,1)&= (-1)^l\bigg(\frac{1+\rho}{2|r_X+r_Y|}\bigg)^l\sum_{j=\lceil l/2\rceil}^{l}\frac{1}{(2j-l)!}\binom{l}{j} \nonumber \\
&\quad\times\bigg(\frac{1}{8}\bigg(\frac{1-\rho}{1+\rho}\bigg)(r_X+r_Y)^2\bigg)^j H_{2j}\!\bigg(\frac{\mathrm{i}}{\sqrt{8}}\sqrt{\frac{1+\rho}{1-\rho}}(r_X-r_Y)\bigg), \label{33b}
\end{align}
where $\mathrm{i}=\sqrt{-1}$.

\vspace{3mm}

\noindent 2. Suppose that $r_X+r_Y\not=0$. Then, as $x\rightarrow\infty$,
\begin{align*}
f_{Z}(x)\sim  \frac{\mathrm{e}^{-r_X^2/2}}{\sqrt{2\pi sx}}\exp\bigg\{-\frac{x}{s(1+\rho)}\bigg\}\sum_{k=0}^\infty d_k(r_X,\rho,1)\frac{s^k}{x^k},  
\end{align*}
where $d_0(r_X,\rho,1)=1$ and, for $k\geq1$,
\begin{align*}
d_k(r_X,\rho,1)=\binom{2k}{k}\bigg(\frac{1-\rho^2}{32}\bigg)^k H_{2k}\bigg(\frac{\mathrm{i}}{\sqrt{2}}\sqrt{\frac{1+\rho}{1-\rho}}\,r_X\bigg).
\end{align*}
\end{corollary}

\begin{remark}
We elected to state the asymptotic expansions of Corollary \ref{corpn2} in terms of the physicist's Hermite polynomials because there is an in-built function in \emph{Mathematica} for the physicist's Hermite polynomials, but not for the probabilist's Hermite polynomials $He_n(x)=(-1)^n\mathrm{e}^{x^2/2}\frac{\mathrm{d}^n}{\mathrm{d}x^n}(\mathrm{e}^{-x^2/2})$. However, given the importance of the probabilist's Hermite polynomials in probability and statistics, and their natural connection to the normal distribution, we note that our formulas can equivalently be stated in terms of these polynomials via the relation $H_n(x)=2^{n/2}He_n(\sqrt{2}x)$
(see \cite[Eq.\ 18.7.12]{olver}).
\end{remark}

\section{Numerical results}\label{sec4}

\begin{table}[h]
	\centering
	\caption{\footnotesize{Relative error in approximating the PDF (\ref{diff}) of the difference $X_1-X_2$ by the asymptotic expansions (\ref{thm3.11}) (for $\lambda_1\not=0$) and (\ref{thm3.12}) (for $\lambda_1=0$) with the leading-order term, with first-order correction, and with second-order correction. Here we set $\beta_1=\beta_2=1$. 
	}}
	\label{table1}
	\footnotesize{
		\begin{tabular}{l*{6}{c}}
			\hline
			& \multicolumn{6}{c}{$x$} \\
			\cmidrule(lr){2-7}
			$(\lambda_1,\lambda_2,\alpha_1,\alpha_2)$ & 2.5 & 5 & 7.5 & 10 & 12.5 & 15 \\
			\hline
			(0,1,0.5,0.5) & 8.2E-02 & 4.5E-02 & 3.1E-02 & 2.4E-02 & 1.9E-02 & 1.6E-02 \\
			(0,1,0.5,0.5) & $-$2.6E-02 & $-$7.6E-03 & $-$3.6E-03 & $-$2.1E-03 & $-$1.4E-03 & $-$9.6E-04 \\
			(0,1,0.5,0.5) & 1.4E-02 & 2.2E-03 & 7.1E-04 & 3.1E-04 & 1.7E-04 & 9.8E-05 \\
			\hline
			(0,1,0.5,1) & 1.2E-01 & 6.7E-02 & 4.6E-02 & 3.5E-02 & 2.9E-02 & 2.4E-02 \\
			(0,1,0.5,1) & $-$4.5E-02 & $-$1.3E-02 & $-$6.1E-03 & $-$3.5E-03 & $-$2.3E-03 & $-$1.6E-03 \\
			(0,1,0.5,1) & 2.7E-02 & 4.1E-03 & 1.3E-03 & 5.8E-04 & 3.0E-04 & 1.8E-04 \\
            \hline   		
			(1,0,0.5,0.5) & $-$1.1E-01 & $-$8.8E-02 & $-$7.6E-02 & $-$6.8E-02 & $-$6.2E-02 & $-$5.7E-02 \\
			(1,0,0.5,0.5) & 3.0E-02 & 1.4E-02 & 8.4E-03 & 5.9E-03 & 4.5E-03 & 3.7E-03 \\
			(1,0,0.5,0.5) & 1.9E-02 & 8.0E-03 & 4.6E-03 & 3.0E-03 & 2.2E-03 & 1.7E-03 \\
			\hline
			(1,0,1,0.5) & $-$1.7E-01 & $-$1.3E-01 & $-$1.1E-01 & $-$9.4E-02 & $-$8.4E-02 & $-$7.8E-02 \\
			(1,0,1,0.5) & $-$9.5E-03 & $-$6.8E-03 & $-$5.2E-03 & $-$4.2E-03 & $-$3.5E-03 & $-$3.1E-03 \\
			(1,0,1,0.5) & 1.2E-02 & 4.5E-03 & 2.5E-03 & 1.6E-03 & 1.2E-03 & 8.9E-04 \\
			\hline
			(1,0,0.5,1) & $-$2.0E-01 & $-$1.7E-01 & $-$1.5E-01 & $-$1.3E-01 & $-$1.2E-01 & $-$1.1E-01 \\
			(1,0,0.5,1) & 4.7E-02 & 1.9E-02 & 1.1E-02 & 7.3E-03 & 5.3E-03 & 4.1E-03 \\
			(1,0,0.5,1) & 4.7E-02 & 1.9E-02 & 1.1E-02 & 7.3E-03 & 5.3E-03 & 4.1E-03 \\
			 \hline
			(1,1,0.5,0.5) & $-$2.1E-01 & $-$1.7E-01 & $-$1.5E-01 & $-$1.3E-01 & $-$1.2E-01 & $-$1.1E-01 \\
			(1,1,0.5,0.5) & 3.7E-02 & 1.3E-02 & 5.9E-03 & 3.2E-03 & 1.8E-03 & 1.1E-03 \\
			(1,1,0.5,0.5) & 5.7E-02 & 2.3E-02 & 1.3E-02 & 8.6E-03 & 6.2E-03 & 4.8E-03 \\
			\hline
			(1,1,1,0.5) & $-$2.9E-01 & $-$2.2E-01 & $-$1.9E-01 & $-$1.6E-01 & $-$1.5E-01 & $-$1.4E-01 \\
			(1,1,1,0.5) & $-$3.2E-02 & $-$2.4E-02 & $-$1.9E-02 & $-$1.5E-02 & $-$1.3E-02 & $-$1.1E-02 \\
			(1,1,1,0.5) & 3.6E-02 & 1.3E-02 & 7.0E-03 & 4.5E-03 & 3.2E-03 & 2.4E-03 \\
			\hline
			(1,1,0.5,1) & $-$2.9E-01 & $-$2.4E-01 & $-$2.1E-01 & $-$1.9E-01 & $-$1.8E-01 & $-$1.6E-01 \\
			(1,1,0.5,1) & 4.3E-02 & 1.1E-02 & 3.6E-03 & 7.1E-04 & $-$6.3E-04 & $-$1.3E-03 \\
			(1,1,0.5,1) & 8.7E-02 & 3.5E-02 & 2.0E-02 & 1.3E-02 & 9.7E-03 & 7.4E-03 \\
			\hline
	\end{tabular}}
\end{table}

\begin{table}[h]
	\centering
	\caption{\footnotesize{Relative error in approximating the PDF (\ref{diff}) of the difference $X_1-X_2$ by the asymptotic expansions (\ref{thm3.11}) (for $\lambda_1\not=0$) and (\ref{thm3.12}) (for $\lambda_1=0$) with $k$-th order correction, where $k=3,5,7,9$. Here we set $\beta_1=\beta_2=1$. 
	}}
	\label{table7}
	\footnotesize{
	\begin{tabular}{l*{6}{c}}
			\hline
			& \multicolumn{6}{c}{$x$} \\
			\cmidrule(lr){2-7}
			$(\lambda_1,\lambda_2,\alpha_1,\alpha_2;k)$ & 2.5 & 5 & 7.5 & 10 & 12.5 & 15 \\
			\hline
			(0,1,0.5,0.5;3) & $-$1.1E-02 & $-$9.0E-04 & $-$2.0E-04 & $-$6.6E-05 & $-$2.8E-05 & $-$1.4E-05 \\
			(0,1,0.5,0.5;5) & $-$1.4E-02 & $-$3.0E-04 & $-$3.0E-05 & $-$5.7E-06 & $-$1.6E-06 & $-$5.5E-07 \\
			(0,1,0.5,0.5;7) & $-$3.4E-02 & $-$1.9E-04 & $-$8.6E-06 & $-$9.5E-07 & $-$1.7E-07 & $-$4.1E-08 \\
			(0,1,0.5,0.5;9) & $-$1.4E-01 & $-$2.0E-04 & $-$4.1E-06 & $-$2.6E-07 & $-$2.9E-08 & $-$5.0E-09 \\
			\hline   		
			(1,0,0.5,0.5;3) & $-$3.8E-03 & $-$3.2E-04 & $-$8.7E-06 & 3.8E-05 & 4.2E-05 & 3.8E-05 \\
			(1,0,0.5,0.5;5) & $-$1.9E-03 & $-$2.9E-04 & $-$8.6E-05 & $-$3.8E-05 & $-$2.0E-05 & $-$1.2E-05 \\
			(1,0,0.5,0.5;7) & $-$2.8E-04 & $-$5.6E-07 & 1.0E-05 & 5.9E-06 & 3.1E-06 & 1.7E-06 \\
			(1,0,0.5,0.5;9) & $-$2.6E-03 & $-$1.6E-04 & $-$2.3E-05 & $-$4.9E-06 & $-$1.4E-06 & $-$4.7E-07 \\
			\hline
			(1,1,0.5,0.5;3) & 2.9E-03 & 2.8E-03 & 1.8E-03 & 1.2E-03 & 8.3E-04 & 6.2E-04 \\
			(1,1,0.5,0.5;5) & $-$1.1E-02 & $-$2.0E-03 & $-$7.1E-04 & $-$3.3E-04 & $-$1.8E-04 & $-$1.1E-04 \\
			(1,1,0.5,0.5;7) & 1.2E-02 & 1.0E-03 & 2.4E-04 & 8.4E-05 & 3.7E-05 & 1.8E-05 \\
			(1,1,0.5,0.5;9) & $-$1.7E-02 & $-$7.7E-04 & $-$1.1E-04 & $-$2.8E-05 & $-$9.2E-06 & $-$3.7E-06 \\
			\hline
			\end{tabular}}
\end{table}

\begin{table}[h]
	\centering
	\caption{\footnotesize{Relative error in approximating the probability $\mathbb{P}(X_1-X_2> x)$ by the asymptotic expansions (\ref{thm3.61}) (for $\lambda_1\not=0$) and (\ref{thm3.52}) (for $\lambda_1=0$) with second-order correction.
	}}
	\label{table2}
	\footnotesize{
		\begin{tabular}{l*{6}{c}}
			\hline
			& \multicolumn{6}{c}{$x$} \\
			\cmidrule(lr){2-7}
			$(\lambda_1,\lambda_2,\alpha_1,\alpha_2,\beta_1,\beta_2)$
			& $q_{0.95}$ & $q_{0.975}$ & $q_{0.99}$ & $q_{0.995}$ & $q_{0.999}$ & $q_{0.9999}$ \\
			\hline
			$(0,1,0.5,1,1)$ & 8.6E+00 & 3.2E+00 & 1.2E+00 & 6.9E-01 & 2.2E-01 & 4.4E-02 \\
			$(0,1,0.5,1,2)$ & 1.8E+01 & 5.4E+00 & 1.8E+00 & 9.6E-01 & 2.9E-01 & 6.3E-02 \\
			$(0,1,0.5,2,1)$ & 5.0E+00 & 2.1E+00 & 9.0E-01 & 5.2E-01 & 1.7E-01 & 3.4E-02 \\
			$(0,1,1,0.5,1,1)$ & 2.1E+01 & 5.9E+00 & 1.9E+00 & 1.0E+00 & 3.1E-01 & 6.3E-02 \\
			$(0,1,1,0.5,1,2)$ & 1.5E+02 & 1.8E+01 & 4.0E+00 & 1.8E+00 & 4.8E-01 & 1.1E-01 \\
			$(0,1,1,0.5,2,1)$ & 7.6E+00 & 2.9E+00 & 1.2E+00 & 6.5E-01 & 2.1E-01 & 3.9E-02 \\
			\hline
            $(1,0,0.5,0.5,1,1)$ & 8.0E-02 & 6.0E-02 & 4.3E-02 & 3.6E-02 & 2.7E-02 & 1.8E-02 \\
			$(1,0,0.5,0.5,1,2)$ & 9.2E-02 & 6.8E-02 & 5.0E-02 & 4.1E-02 & 2.6E-02 & 1.8E-02 \\
			$(1,0,0.5,0.5,2,1)$ & 7.0E-02 & 5.3E-02 & 3.9E-02 & 3.2E-02 & 2.3E-02 & 1.5E-02 \\
			$(1,0,1,0.5,1,1)$ & 2.8E-02 & 2.0E-02 & 1.3E-02 & 1.0E-02 & 5.7E-03 & 3.5E-03 \\
			$(1,0,1,0.5,1,2)$ & 3.2E-02 & 2.2E-02 & 1.5E-02 & 1.2E-02 & 8.2E-03 & 5.5E-03 \\
			$(1,0,1,0.5,2,1)$ & 2.5E-02 & 1.8E-02 & 1.3E-02 & 1.1E-02 & 9.3E-03 & 6.0E-03 \\
			$(1,0,0.5,1,1,1)$ & 1.1E-01 & 7.7E-02 & 5.6E-02 & 4.6E-02 & 3.1E-02 & 1.7E-02 \\
			$(1,0,0.5,1,1,2)$ & 1.4E-01 & 9.9E-02 & 6.9E-02 & 5.6E-02 & 3.8E-02 & 2.6E-02 \\
			$(1,0,0.5,1,2,1)$ & 8.5E-02 & 6.3E-02 & 4.6E-02 & 3.7E-02 & 2.5E-02 & 1.9E-02 \\
			\hline
			$(1,1,0.5,0.5,1,1)$ & 1.2E-01 & 8.4E-02 & 6.0E-02 & 4.8E-02 & 3.2E-02 & 2.8E-02 \\
			$(1,1,0.5,0.5,1,2)$ & 1.5E-01 & 1.0E-01 & 7.0E-02 & 5.6E-02 & 3.6E-02 & 2.2E-02 \\
			$(1,1,0.5,0.5,2,1)$ & 9.4E-02 & 6.9E-02 & 5.0E-02 & 4.1E-02 & 2.7E-02 & 1.2E-02 \\
			$(1,1,1,0.5,1,1)$ & 3.8E-02 & 2.6E-02 & 1.7E-02 & 1.4E-02 & 9.1E-03 & 5.2E-03 \\
			$(1,1,1,0.5,1,2)$ & 4.6E-02 & 3.0E-02 & 1.9E-02 & 1.5E-02 & 8.3E-03 & 3.5E-03 \\
			$(1,1,1,0.5,2,1)$ & 3.2E-02 & 2.3E-02 & 1.6E-02 & 1.3E-02 & 7.5E-03 & 2.9E-03 \\
			$(1,1,0.5,1,1,1)$ & 1.5E-01 & 1.1E-01 & 7.5E-02 & 6.0E-02 & 4.1E-02 & 2.7E-02 \\
			$(1,1,0.5,1,1,2)$ & 2.3E-01 & 1.5E-01 & 9.6E-02 & 7.4E-02 & 4.8E-02 & 2.7E-02 \\
			$(1,1,0.5,1,2,1)$ & 1.1E-01 & 8.1E-02 & 5.9E-02 & 4.9E-02 & 3.4E-02 & 2.4E-02 \\
			\hline
	\end{tabular}}
\end{table}

\begin{table}[h]
	\centering
	\caption{\footnotesize{Relative error in approximating $Q_{X_1-X_2}(p)$ by the asymptotic approximations (\ref{thm3.71}) (for $\lambda_1\not=0)$ and (\ref{thm3.72}) (for $\lambda_1=0)$. 
	}}
	\label{table3}
	\footnotesize{
		\begin{tabular}{l*{6}{c}}
			\hline
			& \multicolumn{6}{c}{$x$} \\
			\cmidrule(lr){2-7}
			$(\lambda_1,\lambda_2,\alpha_1,\alpha_2,\beta_1,\beta_2)$
			& $q_{0.95}$ & $q_{0.975}$ & $q_{0.99}$ & $q_{0.995}$ & $q_{0.999}$ & $q_{0.9999}$ \\
            \hline
			$(0,1,0.5,0.5,1,1)$ & $-$1.1E-01 & $-$6.3E-02 & $-$3.8E-02 & $-$2.8E-02 & $-$1.6E-02 & $-$9.4E-03 \\
			$(0,1,0.5,0.5,1,2)$ & $-$2.3E-01 & $-$1.2E-01 & $-$6.3E-02 & $-$4.4E-02 & $-$2.4E-02 & $-$1.3E-02 \\
			$(0,1,0.5,0.5,2,1)$ & $-$5.6E-02 & $-$3.7E-02 & $-$2.4E-02 & $-$1.9E-02 & $-$1.2E-02 & $-$6.8E-03 \\
			$(0,1,0.5,1,1,1)$ & $-$1.8E-01 & $-$9.1E-02 & $-$5.0E-02 & $-$3.5E-02 & $-$1.9E-02 & $-$1.1E-02 \\
			$(0,1,0.5,1,1,2)$ & $-$7.1E-01 & $-$2.3E-01 & $-$1.0E-01 & $-$6.7E-02 & $-$3.3E-02 & $-$1.6E-02 \\
			$(0,1,0.5,1,2,1)$ & $-$7.2E-02 & $-$4.5E-02 & $-$2.8E-02 & $-$2.1E-02 & $-$1.3E-02 & $-$8.0E-03 \\
			\hline
			$(1,0,0.5,0.5,1,1)$ & 3.2E-01 & 2.5E-01 & 1.9E-01 & 1.6E-01 & 1.2E-01 & 8.5E-02 \\
			$(1,0,0.5,0.5,1,2)$ & 2.8E-01 & 2.1E-01 & 1.6E-01 & 1.3E-01 & 9.3E-02 & 6.6E-02 \\
			$(1,0,0.5,0.5,2,1)$ & 2.4E-01 & 1.8E-01 & 1.4E-01 & 1.2E-01 & 8.4E-02 & 6.0E-02 \\
			$(1,0,1,0.5,1,1)$ & 2.1E-01 & 1.7E-01 & 1.4E-01 & 1.2E-01 & 8.9E-02 & 6.7E-02 \\
			$(1,0,1,0.5,1,2)$ & 1.8E-01 & 1.4E-01 & 1.1E-01 & 9.1E-02 & 6.7E-02 & 4.9E-02 \\
			$(1,0,1,0.5,2,1)$ & 1.5E-01 & 1.2E-01 & 9.1E-02 & 7.9E-02 & 6.0E-02 & 4.4E-02 \\
			$(1,0,0.5,1,1,1)$ & 5.4E-01 & 4.1E-01 & 3.2E-01 & 2.7E-01 & 2.0E-01 & 1.4E-01 \\
			$(1,0,0.5,1,1,2)$ & 4.8E-01 & 3.5E-01 & 2.6E-01 & 2.1E-01 & 1.5E-01 & 1.1E-01 \\
			$(1,0,0.5,1,2,1)$ & 3.5E-01 & 2.7E-01 & 2.1E-01 & 1.7E-01 & 1.3E-01 & 9.4E-02 \\
			\hline
			$(1,1,0.5,0.5,1,1)$ & 4.0E-01 & 3.0E-01 & 2.2E-01 & 1.8E-01 & 1.3E-01 & 9.5E-02 \\
			$(1,1,0.5,0.5,1,2)$ & 4.2E-01 & 3.0E-01 & 2.1E-01 & 1.7E-01 & 1.2E-01 & 8.1E-02 \\
			$(1,1,0.5,0.5,2,1)$ & 2.7E-01 & 2.0E-01 & 1.5E-01 & 1.3E-01 & 9.0E-02 & 6.3E-02 \\
			$(1,1,1,0.5,1,1)$ & 2.7E-01 & 2.1E-01 & 1.6E-01 & 1.4E-01 & 1.0E-01 & 7.5E-02 \\
			$(1,1,1,0.5,1,2)$ & 2.7E-01 & 2.0E-01 & 1.5E-01 & 1.2E-01 & 8.9E-02 & 6.3E-02 \\
			$(1,1,1,0.5,2,1)$ & 1.7E-01 & 1.3E-01 & 1.0E-01 & 8.7E-02 & 6.5E-02 & 4.7E-02 \\
			$(1,1,0.5,1,1,1)$ & 6.7E-01 & 4.9E-01 & 3.6E-01 & 3.0E-01 & 2.2E-01 & 1.6E-01 \\
			$(1,1,0.5,1,1,2)$ & 7.3E-01 & 4.9E-01 & 3.4E-01 & 2.7E-01 & 1.9E-01 & 1.3E-01 \\
			$(1,1,0.5,1,2,1)$ & 4.0E-01 & 3.0E-01 & 2.2E-01 & 1.9E-01 & 1.4E-01 & 9.9E-02 \\
			\hline
	\end{tabular}}
\end{table}

\begin{table}[h]
	\centering
	\caption{\footnotesize{Relative error in approximating the PDF \eqref{sum} of the sum $X_1+X_2$ by the asymptotic expansions (\ref{thm3.11}) (for $\lambda_1\not=0$) and (\ref{thm3.12}) (for $\lambda_1=0$) with the leading-order term, with first-order correction, and with second-order correction. Here we set $\beta_1=2$ and $\beta_2=1$. 
	}}
	\label{table4}
	\footnotesize{
	\begin{tabular}{l*{6}{c}}
		\hline
		& \multicolumn{6}{c}{$x$} \\
		\cmidrule(lr){2-7}
		$(\lambda_1,\lambda_2,\alpha_1,\alpha_2)$ & 10 & 20 & 30 & 40 & 50 & 60 \\
		\hline
		(0,1,0.5,0.5) & $-$2.7E-01 & $-$1.7E-01 & $-$1.0E-01 & $-$7.1E-02 & $-$5.5E-02 & $-$4.5E-02 \\
		(0,1,0.5,0.5) & $-$8.4E-02 & $-$7.0E-02 & $-$2.7E-02 & $-$1.3E-02 & $-$7.7E-03 & $-$5.2E-03 \\
		(0,1,0.5,0.5) & 3.4E-02 & $-$3.7E-02 & $-$1.0E-02 & $-$3.6E-03 & $-$1.6E-03 & $-$8.9E-04 \\ \hline
		
		(0,1,1,0.5) & 1.4E-01 & 6.8E-03 & 2.5E-04 & 7.6E-06 & 2.0E-07 & 4.6E-09 \\
		(0,1,1,0.5) & 1.4E-01 & 6.8E-03 & 2.5E-04 & 7.6E-06 & 2.0E-07 & 4.6E-09 \\
		(0,1,1,0.5) & 1.4E-01 & 6.8E-03 & 2.5E-04 & 7.6E-06 & 2.0E-07 & 4.6E-09 \\ \hline
		
		(0,1,0.5,1) & $-$2.9E-01 & $-$2.1E-01 & $-$1.2E-01 & $-$8.6E-02 & $-$6.6E-02 & $-$5.4E-02 \\
		(0,1,0.5,1) & $-$7.7E-02 & $-$8.9E-02 & $-$3.5E-02 & $-$1.7E-02 & $-$1.0E-02 & $-$6.7E-03 \\
		(0,1,0.5,1) & 7.3E-02 & $-$4.8E-02 & $-$1.5E-02 & $-$5.0E-03 & $-$2.3E-03 & $-$1.2E-03 \\ \hline
		
		(1,0,0.5,0.5) & 1.7E-01 & 1.3E-01 & 1.1E-01 & 9.6E-02 & 8.7E-02 & 8.1E-02 \\
		(1,0,0.5,0.5) & $-$9.4E-02 & $-$5.0E-02 & $-$3.5E-02 & $-$2.6E-02 & $-$2.1E-02 & $-$1.8E-02 \\
		(1,0,0.5,0.5) & 5.2E-02 & 2.0E-02 & 1.2E-02 & 7.8E-03 & 5.7E-03 & 4.4E-03 \\ \hline
		
		(1,0,1,0.5) & 1.5E-01 & 1.1E-01 & 9.7E-02 & 8.6E-02 & 7.7E-02 & 7.1E-02 \\
		(1,0,1,0.5) & $-$7.4E-02 & $-$3.9E-02 & $-$2.7E-02 & $-$2.1E-02 & $-$1.7E-02 & $-$1.4E-02 \\
		(1,0,1,0.5) & 3.8E-02 & 1.5E-02 & 8.5E-03 & 5.7E-03 & 4.2E-03 & 3.2E-03 \\ \hline
		
		(1,0,0.5,1) & 3.7E-01 & 2.8E-01 & 2.3E-01 & 2.0E-01 & 1.8E-01 & 1.7E-01 \\
		(1,0,0.5,1) & $-$2.4E-01 & $-$1.3E-01 & $-$8.7E-02 & $-$6.6E-02 & $-$5.4E-02 & $-$4.5E-02 \\
		(1,0,0.5,1) & 1.7E-01 & 6.4E-02 & 3.6E-02 & 2.4E-02 & 1.7E-02 & 1.3E-02 \\ \hline
		
		(1,1,0.5,0.5) & 1.1E+00 & 7.7E-01 & 6.3E-01 & 5.5E-01 & 4.9E-01 & 4.5E-01 \\
		(1,1,0.5,0.5) & $-$1.2E+00 & $-$6.3E-01 & $-$4.2E-01 & $-$3.2E-01 & $-$2.6E-01 & $-$2.1E-01 \\
		(1,1,0.5,0.5) & 1.5E+00 & 5.4E-01 & 3.0E-01 & 1.9E-01 & 1.4E-01 & 1.1E-01 \\ \hline
		
		(1,1,1,0.5) & 1.2E+00 & 8.1E-01 & 6.5E-01 & 5.6E-01 & 5.0E-01 & 4.5E-01 \\
		(1,1,1,0.5) & $-$1.2E+00 & $-$5.9E-01 & $-$3.9E-01 & $-$2.9E-01 & $-$2.3E-01 & $-$1.9E-01 \\
		(1,1,1,0.5) & 1.4E+00 & 4.7E-01 & 2.6E-01 & 1.7E-01 & 1.2E-01 & 9.1E-02 \\ \hline
		
		(1,1,0.5,1) & 1.5E+00 & 1.0E+00 & 8.2E-01 & 7.0E-01 & 6.2E-01 & 5.7E-01 \\
		(1,1,0.5,1) & $-$1.8E+00 & $-$9.0E-01 & $-$5.9E-01 & $-$4.4E-01 & $-$3.5E-01 & $-$2.9E-01 \\
		(1,1,0.5,1) & 2.4E+00 & 8.1E-01 & 4.4E-01 & 2.8E-01 & 2.0E-01 & 1.5E-01 \\
		\hline
		\end{tabular}}
\end{table}

\begin{table}[h]
	\centering
	\caption{\footnotesize{Relative error in approximating the PDF \eqref{sum} of the sum $X_1+X_2$ by the asymptotic expansions (\ref{thm3.11}) (for $\lambda_1\not=0$) and (\ref{thm3.12}) (for $\lambda_1=0$) with $k$-th order correction, where $k=3,5,7,9$. Here we set $\beta_1=2$ and $\beta_2=1$.
	}}
	\label{table8}
	\footnotesize{
		\begin{tabular}{l*{6}{c}}
			\hline
			& \multicolumn{6}{c}{$x$} \\
			\cmidrule(lr){2-7}
			$(\lambda_1,\lambda_2,\alpha_1,\alpha_2;k)$ & 10 & 20 & 30 & 40 & 50 & 60 \\
			\hline
			(0,1,0.5,0.5;3) & 1.5E-01 & $-$2.0E-02 & $-$5.2E-03 & $-$1.3E-03 & $-$4.5E-04 & $-$2.0E-04 \\
			(0,1,0.5,0.5;5) & 5.1E-01 & $-$2.4E-03 & $-$1.9E-03 & $-$3.1E-04 & $-$6.2E-05 & $-$1.8E-05 \\
			(0,1,0.5,0.5;7) & 1.7E+00 & 1.2E-02 & $-$7.7E-04 & $-$1.2E-04 & $-$1.6E-05 & $-$2.9E-06 \\
			(0,1,0.5,0.5;9) & 8.7E+00 & 3.2E-02 & $-$9.1E-05 & $-$5.5E-05 & $-$5.8E-06 & $-$7.2E-07 \\
			\hline   		
			(1,0,0.5,0.5;3) & $-$3.9E-02 & $-$1.1E-02 & $-$5.0E-03 & $-$2.9E-03 & $-$1.9E-03 & $-$1.3E-03 \\
			(1,0,0.5,0.5;5) & $-$2.8E-02 & $-$3.9E-03 & $-$1.2E-03 & $-$5.3E-04 & $-$2.8E-04 & $-$1.6E-04 \\
			(1,0,0.5,0.5;7) & $-$2.7E-02 & $-$1.9E-03 & $-$4.0E-04 & $-$1.3E-04 & $-$5.6E-05 & $-$2.7E-05 \\
			(1,0,0.5,0.5;9) & $-$3.3E-02 & $-$1.2E-03 & $-$1.7E-04 & $-$4.1E-05 & $-$1.4E-05 & $-$5.7E-06 \\
			\hline
			(1,1,0.5,0.5;3) & $-$1.9E+00 & $-$4.9E-01 & $-$2.2E-01 & $-$1.2E-01 & $-$8.0E-02 & $-$5.6E-02 \\
			(1,1,0.5,0.5;5) & $-$3.6E+00 & $-$4.6E-01 & $-$1.4E-01 & $-$5.9E-02 & $-$3.0E-02 & $-$1.8E-02 \\
			(1,1,0.5,0.5;7) & $-$7.7E+00 & $-$5.0E-01 & $-$1.0E-01 & $-$3.2E-02 & $-$1.3E-02 & $-$6.4E-03 \\
			(1,1,0.5,0.5;9) & $-$1.9E+01 & $-$6.1E-01 & $-$8.3E-02 & $-$2.0E-02 & $-$6.6E-03 & $-$2.7E-03 \\
			\hline
			\end{tabular}}
\end{table}

\begin{table}[h]
	\centering
	\caption{\footnotesize{Relative error in approximating the probability $\mathbb{P}(X_1+X_2> x)$ by the asymptotic expansions (\ref{thm3.61}) (for $\lambda_1\not=0$) and (\ref{thm3.52}) (for $\lambda_1=0$) with second-order correction.
	}}
	\label{table5}
	\footnotesize{
		\begin{tabular}{l*{6}{c}}
			\hline
			& \multicolumn{6}{c}{$x$} \\
			\cmidrule(lr){2-7}
			$(\lambda_1,\lambda_2,\alpha_1,\alpha_2,\beta_1,\beta_2)$ 
			& $q_{0.95}$ & $q_{0.975}$ & $q_{0.99}$ & $q_{0.995}$ & $q_{0.999}$ & $q_{0.9999}$ \\
			\hline
			(0,1,0.5,0.5,2,1) & 5.4E-01 & 2.7E-01 & 1.0E-01 & 3.2E-02 & $-$3.7E-02 & $-$5.3E-02 \\
			(0,1,0.5,0.5,4,1) & 9.0E-01 & 5.3E-01 & 2.8E-01 & 1.7E-01 & 5.6E-02 & 1.7E-04 \\
			(0,1,1,0.5,2,1) & 2.5E-01 & 1.4E-01 & 5.4E-02 & 1.9E-02 & $-$2.1E-02 & $-$3.2E-02 \\
			(0,1,1,0.5,4,1) & 2.1E-01 & 1.1E-01 & 3.6E-02 & 6.8E-03 & $-$2.6E-02 &$-$3.7E-02 \\
			(0,1,0.5,1,2,1) & 4.5E-01 & 2.2E-01 & 6.7E-02 & 6.2E-03 & $-$5.3E-02 & $-$6.4E-02 \\
			(0,1,0.5,1,4,1) & 7.2E-01 & 4.3E-01 & 2.3E-01 & 1.5E-01 & 4.4E-02 & $-$7.6E-03 \\ \hline
			
			(1,0,0.5,0.5,2,1) & 5.9E-02 & 4.6E-02 & 3.5E-02 & 3.0E-02 & 2.3E-02 & 1.7E-02 \\
			(1,0,0.5,0.5,4,1) & 5.1E-02 & 3.9E-02 & 2.9E-02 & 2.4E-02 & 1.8E-02 & 1.3E-02 \\
			(1,0,1,0.5,2,1) & 2.5E-02 & 1.9E-02 & 1.5E-02 & 1.2E-02 & 7.3E-03 & 3.9E-03 \\
			(1,0,1,0.5,4,1) & 1.9E-02 & 1.4E-02 & 9.7E-03 & 8.0E-03 & 5.8E-03 & 4.6E-03 \\
			(1,0,0.5,1,2,1) & 8.4E-02 & 6.7E-02 & 5.2E-02 & 4.5E-02 & 3.5E-02 & 2.2E-02 \\
			(1,0,0.5,1,4,1) & 4.7E-02 & 3.6E-02 & 2.7E-02 & 2.2E-02 & 1.6E-02 & 6.9E-03 \\ \hline
			
			(1,1,0.5,0.5,2,1) & 5.7E-01 & 4.6E-01 & 3.6E-01 & 3.1E-01 & 2.3E-01 & 1.6E-01 \\
			(1,1,0.5,0.5,4,1) & 5.0E-02 & 3.9E-02 & 3.0E-02 & 2.5E-02 & 1.9E-02 & 9.2E-03 \\
			(1,1,1,0.5,2,1) & 4.0E-01 & 3.3E-01 & 2.6E-01 & 2.3E-01 & 1.8E-01 & 1.3E-01 \\
			(1,1,1,0.5,4,1) & 2.0E-02 & 1.6E-02 & 1.1E-02 & 9.4E-03 & 6.6E-03 & 6.5E-03 \\
			(1,1,0.5,1,2,1) & 8.2E-01 & 6.6E-01 & 5.2E-01 & 4.5E-01 & 3.3E-01 & 2.4E-01 \\
			(1,1,0.5,1,4,1) & 5.1E-02 & 4.0E-02 & 3.1E-02 & 2.6E-02 & 1.8E-02 & 1.2E-02 \\ \hline
	\end{tabular}}
\end{table}

\begin{table}[h]
	\centering
	\caption{\footnotesize{Relative error in approximating $Q_{X_1+X_2}(p)$ by the asymptotic approximations (\ref{thm3.71}) (for $\lambda_1\not=0)$ and (\ref{thm3.72}) (for $\lambda_1=0)$. 
	}}
	\label{table6}
	\footnotesize{
		\begin{tabular}{l*{6}{c}}
			\hline
			& \multicolumn{6}{c}{$x$} \\
			\cmidrule(lr){2-7}
			$(\lambda_1,\lambda_2,\alpha_1,\alpha_2,\beta_1,\beta_2)$ 
			& $q_{0.95}$ & $q_{0.975}$ & $q_{0.99}$ & $q_{0.995}$ & $q_{0.999}$ & $q_{0.9999}$ \\
			\hline
			
			(0,1,0.5,0.5,2,1) & $-$3.1E-02 & $-$4.0E-02 & $-$4.0E-02 & $-$3.7E-02 & $-$2.8E-02 & $-$1.7E-02 \\
			(0,1,0.5,0.5,4,1) & $-$3.9E-02 & $-$2.8E-02 & $-$1.9E-02 & $-$1.5E-02 & $-$9.8E-03 & $-$5.7E-03 \\
			(0,1,1,0.5,2,1) & 3.1E-02 & 1.7E-02 & 8.4E-03 & 5.0E-03 & 1.6E-03 & 8.3E-04 \\
			(0,1,1,0.5,4,1) & 2.1E-04 & 5.4E-05 & 8.9E-05 & 1.0E-05 & $-$7.2E-05 & 2.9E-04 \\
			(0,1,0.5,1,2,1) & $-$2.7E-02 & $-$3.7E-02 & $-$3.9E-02 & $-$3.7E-02 & $-$2.8E-02 & $-$1.8E-02 \\
			(0,1,0.5,1,4,1) & $-$4.0E-02 & $-$2.9E-02 & $-$2.0E-02 & $-$1.6E-02 & $-$1.0E-02 & $-$6.3E-03 \\ \hline
			
			(1,0,0.5,0.5,2,1) & 1.4E-01 & 1.0E-01 & 7.6E-02 & 6.3E-02 & 4.4E-02 & 3.0E-02 \\
			(1,0,0.5,0.5,4,1) & 1.4E-01 & 1.0E-01 & 7.3E-02 & 6.0E-02 & 4.2E-02 & 2.8E-02 \\
			(1,0,1,0.5,2,1) & 6.8E-02 & 5.2E-02 & 3.9E-02 & 3.3E-02 & 2.3E-02 & 1.6E-02 \\
			(1,0,1,0.5,4,1) & 6.3E-02 & 4.8E-02 & 3.6E-02 & 3.0E-02 & 2.1E-02 & 1.4E-02 \\
			(1,0,0.5,1,2,1) & 1.4E-01 & 1.1E-01 & 7.9E-02 & 6.5E-02 & 4.6E-02 & 3.1E-02 \\
			(1,0,0.5,1,4,1) & 1.3E-01 & 1.0E-01 & 7.3E-02 & 6.0E-02 & 4.1E-02 & 2.7E-02 \\ \hline
			
			(1,1,0.5,0.5,2,1) & 1.7E-01 & 1.3E-01 & 1.0E-01 & 8.4E-02 & 6.0E-02 & 4.1E-02 \\
			(1,1,0.5,0.5,4,1) & 1.4E-01 & 1.0E-01 & 7.5E-02 & 6.2E-02 & 4.3E-02 & 2.9E-02 \\
			(1,1,1,0.5,2,1) & 1.1E-01 & 8.4E-02 & 6.5E-02 & 5.5E-02 & 4.1E-02 & 2.9E-02 \\
			(1,1,1,0.5,4,1) & 6.6E-02 & 5.1E-02 & 3.8E-02 & 3.2E-02 & 2.3E-02 & 1.6E-02 \\
			(1,1,0.5,1,2,1) & 1.7E-01 & 1.3E-01 & 1.0E-01 & 8.5E-02 & 6.1E-02 & 4.3E-02 \\
			(1,1,0.5,1,4,1) & 1.3E-01 & 1.0E-01 & 7.4E-02 & 6.1E-02 & 4.3E-02 & 2.9E-02 \\
			\hline
	\end{tabular}}
\end{table}

In this section, we provide numerical results to assess the quality of the asymptotic approximations given in Section \ref{sec3} for the PDF, tail probabilities and quantile function for the difference $X_1-X_2$ (Tables \ref{table1}--\ref{table3}) and the sum $X_1+X_2$ (Tables \ref{table4}--\ref{table6}) across a range of parameter constellations.

In Tables \ref{table1} and \ref{table7}, we give the relative error in approximating the PDF of the non-central gamma difference distribution by the asymptotic expansions \eqref{thm3.11} (when $\lambda_1\neq0$) and \eqref{thm3.12} (when $\lambda_1=0$). In compiling the figures for this table, we used \emph{Mathematica} to evaluate the PDF of $X_1-X_2$ by truncation of the infinite series in formulas (\ref{diff}) (when $\lambda_1\neq0$) and (\ref{diff0}) (when $\lambda_1=0$) at $k=50$, which allowed the PDF to be computed to a high level of accuracy. In Table \ref{table1} and all other tables, a negative figure means that the approximation gives a smaller value than the true value. 

 In Table \ref{table2}, we report the relative error in approximating the tail probability $\mathbb{P}(X_1-X_2>x)$ by the asymptotic expansions \eqref{thm3.61} (when $\lambda_1\neq0$) and \eqref{thm3.52} (when $\lambda_1=0$) with second-order correction. Here we took $x$ to be a quantile $q_p$ for $p=0.95,\ldots,0.9999$; we do not report results for the same values of $x$ as in Tables \ref{table1} and \ref{table7} because we obtained similar results (although in some instances the relative errors are not accurate to two significant figures because the error from the asymptotic approximation is non-negligible compared to the simulation error), as would be expected since the asymptotic expansions for the tail probabilities are derived from those for the PDF. We obtained the quantile values (which were also used in Table \ref{table3}) by running Monte Carlo simulations with $10^9$ realisations using \emph{Mathematica}. For the parameter constellations considered, the random variables $X_1$ and $X_2$ both follow scaled non-central chi-square distributions or scaled chi-square distributions, and such distributions can be simulated in \emph{Mathematica} by using the built-in \emph{RandomVariate} function combined with the \emph{NoncentralChiSquareDistribution} function. In Table \ref{table3}, we provide the relative error in approximating the quantile function $Q_{X_1-X_2}(p)$ by the asymptotic approximations \eqref{thm3.71} (when $\lambda_1\neq0$) and \eqref{thm3.72} (when $\lambda_1=0$). Analogues of Tables \ref{table1}--\ref{table3} are given for the sum $X_1+X_2$ in Tables \ref{table4}--\ref{table6}, and the results were obtained in an exactly analogous manner.

Results are reported to two significant figures. Most figures in the tables are accurate to two significant figures; however, this is not the case for some of the entries in Tables \ref{table2}, \ref{table3}, \ref{table5} and \ref{table6} when $p=0.999$ and $p=0.9999$ with more variation for $p=0.9999$ (which we observed by repeating the simulations several times). This is because in these cases the simulation error is non-negligible relative to the error arising from application of the asymptotic approximations.

From Table \ref{table1}, we observe that the relative errors resulting from applications of the asymptotic approximation (\ref{thm3.12}) (that is when $\lambda_1=0$) typically decay at a faster rate than the relative errors from applications of the asymptotic approximation (\ref{thm3.11}) (when $\lambda_1\not=0$). This is to be expected since, for example, the relative error from applying the leading-order term from the asymptotic expansion (\ref{thm3.12}) is of order $x^{-1}$, whilst the relative error for the leading-order term from the asymptotic expansion (\ref{thm3.11}) is of order $x^{-1/2}$. Also, as expected, we see that the relative errors tend to decrease for a given value of $x$ and a given parameter constellation as one includes further correction terms in the asymptotic expansions. The asymptotic approximations are rather accurate even for small values of $x$; indeed, when applying the asymptotic expansions with second-order correction the largest relative error amongst all parameter constellations we considered is just 8.7\% when $x=2.5$ and only 0.74\% when $x=15$. Similar conclusions can be drawn from Table \ref{table4}, although the relative errors reported tend to be higher than those given in Table \ref{table1}.
From Table \ref{table4}, we observe that for the parameter values $(\lambda_1,\lambda_2,\alpha_1,\alpha_2,\beta_1,\beta_2)=(0,1,0.5,0.5,2,1)$ the relative errors in the asymptotic approximation using the leading-order term, first-order correction and second-order correction are the same; this is because for these parameter values we have $c_l=0$ for $l\geq1$. We also observe that for this parameter constellation the relative error decreases at a rate that is consistent with exponential decay in $x$, which we claimed would be expected in Remark \ref{remexp}. In Tables \ref{table7} and \ref{table8}, we tested the effect of including further terms in the asymptotic expansions for the PDF of the non-central gamma difference and sum distributions. We find that for larger values of $x$ retaining further terms results in exceptionally accurate asymptotic approximations, whilst for smaller values of $x$ retaining further terms can make the approximations less accurate.

In Tables \ref{table2} and \ref{table5}, we only report results for the relative error in the asymptotic approximations of the tail probabilities using a second-order correction; our numerical experiments (that we do not report) revealed that, as was the case in Tables \ref{table1} and \ref{table4}, using a second-order correction typically yields more accurate results than using just the leading-order term or first-order correction. In Table \ref{table2}, the asymptotic approximation performs poorly when $\lambda_1=0$ particularly for the quantiles $q_{0.95}$, $q_{0.975}$, $q_{0.99}$ and $q_{0.995}$. This can be explained from the fact that for the parameter constellations with $\lambda_1=0$ the quantiles were positive (and so the asymptotic approximation is valid) but quite small, meaning that one would not expect reasonable accuracy from the asymptotic approximation. In contrast, from Tables \ref{table3} and \ref{table6} we see that the asymptotic approximations for the quantile function deliver reasonable accuracy across the range of parameter values, particularly when $\lambda_1=0$, which is to be expected given that the error in the asymptotic approximation (\ref{thm3.72}) is of the faster order $O(1/\ln(1/(1-p)))$ than the $O(1/\sqrt{\ln(1/(1-p))})$ rate of the asymptotic approximation (\ref{thm3.71}). 

\section{Proofs}\label{sec5}

In proving Theorems \ref{thm1.1} and \ref{thm2.3} we will require several definite integral formulas which we now state. In deriving the series representation (\ref{diff}) we will use the integral formula
\begin{equation}\label{int1}
\int_{-\infty}^\infty(z+\mathrm{i}t)^{-\rho}(y-\mathrm{i}t)^{-\sigma}\mathrm{e}^{\mathrm{i}\lambda t}\,\mathrm{d}t=\frac{2\pi}{\Gamma(\delta)}(y+z)^{1-\rho-\sigma}\mathrm{e}^{-|\lambda|\theta}U(1-\delta,2-\rho-\sigma,|\lambda|(y+z)),
\end{equation}
where $y,z,\rho,\sigma>0$, $\delta=\sigma$, $\theta=y$ for $\lambda<0$, and $\delta=\rho$, $\theta=z$ for $\lambda\geq0$. Integral formula (\ref{int1}) is a correction of the integral formula of \cite[p.\ 325, Eq.\ 19]{integralbook}, with the correction given by \cite[Eq.\ (A.35)]{gnp24}. In \cite{integralbook}, the integral formula was stated under the stronger condition that $\rho+\sigma>1$ (which ensures that the integrand is Lebesgue integrable over $\mathbb{R}$). However, if the integral in equation (\ref{int1}) is understood to exist in the improper Riemann sense (which is sufficient for our purposes), then the integral exists under the weaker condition that $\rho,\sigma>0$.

To obtain the series representation (\ref{sum}) we will make use of the following definite integral formula (see \cite[p.\ 325, Eq.\ 18]{integralbook}):
\begin{equation}\label{int2}
	\int_{-\infty}^{\infty}(z-\mathrm{i}t)^{-\rho}(y-\mathrm{i}t)^{-\sigma}\mathrm{e}^{\mathrm{i}\lambda t}\,\mathrm{d}t=\frac{2\pi(-\lambda)^{\rho+\sigma-1}}{\Gamma(\rho+\sigma)}\mathrm{e}^{\lambda z}M(\sigma;\rho+\sigma,-\lambda(z-y)), 
\end{equation}
where $\lambda<0$ and $y,z,\rho,\sigma>0$. The integral formula (\ref{int2}) is again stated in \cite{integralbook} under the assumption that $\rho+\sigma>1$; however, if the integral is understood to exist in the improper Riemann sense, then it exists under the weaker condition $\rho,\sigma>0$.

Finally, we will make use of the following integral, which is a special case of \cite[Eq.\ 3.384(9)]{g07}:
\begin{align}\label{int3}
	\int_{-\infty}^{\infty}(z+\mathrm{i}t)&^{-2\nu}(y-\mathrm{i}t)^{-2\nu}\mathrm{e}^{\mathrm{i}\lambda t}\,\mathrm{d}t\nonumber\\
    &=2\pi(y+z)^{-2\nu}\frac{|\lambda|^{2\nu-1}}{\Gamma(2\nu)}\exp\bigg(-\frac{z-y}{2}\lambda\bigg)W_{0,\frac{1}{2}-2\nu}((y+z) |\lambda|),
\end{align}
where $y,z,\nu>0$, $\lambda\in\mathbb{R}$, and $W_{\kappa,\mu}(x)$ is the Whittaker function (see \cite[Chapter 13]{olver}). Again, the integral formula is given in \cite{g07} under the assumption that $\nu>-1/4$, but this assumption can be weakened to $\nu>0$ if the integral is understood to exist in the improper Riemann sense.

\begin{proof}[Proof of Theorem \ref{thm1.1}] Since $X_1$ and $X_2$ are independent, the characteristic function of the difference $X_1-X_2$ is given by
	\begin{align}
		\phi_{X_1-X_2}(t)&=\phi_{X_1}(t)\phi_{X_2}(t)\nonumber \\
        &=(1-\mathrm{i}\beta_1t)^{-\alpha_1}(1+\mathrm{i}\beta_2t)^{-\alpha_2}\exp\bigg(-\lambda_1-\lambda_2+\frac{\lambda_1}{1-\mathrm{i}\beta_1t}+\frac{\lambda_2}{1+\mathrm{i}\beta_2t}\bigg),\label{cfdiff}
	\end{align}
for $t\in\mathbb{R}$. Upon applying the Taylor series expansion for the exponential function and the binomial theorem we obtain that 
\begin{align}
	\phi_{X_1-X_2}(t)
	&=\frac{\mathrm{e}^{-\lambda_1-\lambda_2}}{(1-\mathrm{i}\beta_1t)^{\alpha_1}(1+\mathrm{i}\beta_2t)^{\alpha_2}}\sum_{k=0}^{\infty}\frac{1}{k!}\bigg(\frac{\lambda_1}{1-\mathrm{i}\beta_1t}+\frac{\lambda_2}{1+\mathrm{i}\beta_2t}\bigg)^k\nonumber\\
	&=\mathrm{e}^{-\lambda_1-\lambda_2}\sum_{k=0}^{\infty}\sum_{j=0}^{k}\frac{1}{k!}\binom{k}{j}\lambda_1^{k-j}\lambda_2^j(1-\mathrm{i}\beta_1t)^{j-k-\alpha_1}(1+\mathrm{i}\beta_2t)^{-j-\alpha_2}.\nonumber
\end{align}
Applying the inversion theorem now yields the following expression for the PDF of $X_1-X_2$:
\begin{align}
	f_{X_1-X_2}(x)	&=\frac{1}{2\pi}\mathrm{e}^{-\lambda_1-\lambda_2}\sum_{k=0}^{\infty}\sum_{j=0}^{k}\frac{1}{k!}\binom{k}{j}\lambda_1^{k-j}\lambda_2^j\beta_1^{j-k-\alpha_1}\beta_2^{-j-\alpha_2}\nonumber\\
	&\quad\times\int_{-\infty}^{\infty}\mathrm{e}^{-\mathrm{i}tx}\bigg(\frac{1}{\beta_1}-\mathrm{i}t\bigg)^{j-k-\alpha_1}\bigg(\frac{1}{\beta_2}+\mathrm{i}t\bigg)^{-j-\alpha_2}\,\mathrm{d}t.\nonumber
\end{align}
Upon evaluating the integral using the integral formula (\ref{int1}), we have, for $x\geq0$,
\begin{align}
	f_{X_1-X_2}(x)&=\mathrm{e}^{-\lambda_1-\lambda_2-x/\beta_1}\sum_{k=0}^{\infty}\sum_{j=0}^{k}\frac{1}{k!\Gamma(\alpha_1+k-j)}\binom{k}{j}\lambda_1^{k-j}\lambda_2^j\beta_1^{j-k-\alpha_1}\beta_2^{-j-\alpha_2}\nonumber\\
	&\quad\times\bigg(\frac{1}{\beta_1}+\frac{1}{\beta_2}\bigg)^{1-\alpha_1-\alpha_2-k}U\bigg(1-\alpha_1-k+j,2-\alpha_1-\alpha_2-k,\bigg(\frac{1}{\beta_1}+\frac{1}{\beta_2}\bigg)x\bigg),\nonumber
\end{align}
and, for $x<0$,
\begin{align}
	f_{X_1-X_2}(x)&=\mathrm{e}^{-\lambda_1-\lambda_2-x/\beta_2}\sum_{k=0}^{\infty}\sum_{j=0}^{k}\frac{1}{k!\Gamma(j+\alpha_2)}\binom{k}{j}\lambda_1^{k-j}\lambda_2^j\beta_1^{j-k-\alpha_1}\beta_2^{-j-\alpha_2}\nonumber\\
	&\quad\times\bigg(\frac{1}{\beta_1}+\frac{1}{\beta_2}\bigg)^{1-\alpha_1-\alpha_2-k}U\bigg(1-\alpha_2-j,2-\alpha_1-\alpha_2-k,-\bigg(\frac{1}{\beta_1}+\frac{1}{\beta_2}\bigg)x\bigg).\nonumber
\end{align}
This completes the derivation of the series representation (\ref{diff}).

The derivation of the series representation \eqref{sum} for the PDF of the sum $X_1+X_2$ is similar. The characteristic function of $X_1+X_2$ can be expressed as
\begin{align}
	\phi_{X_1+X_2}(t)&=(1-\mathrm{i}\beta_1t)^{-\alpha_1}(1-\mathrm{i}\beta_2t)^{-\alpha_2}\exp\bigg(-\lambda_1-\lambda_2+\frac{\lambda_1}{1-\mathrm{i}\beta_1t}+\frac{\lambda_2}{1-\mathrm{i}\beta_2t}\bigg)\nonumber \\
&=\mathrm{e}^{-\lambda_1-\lambda_2}\sum_{k=0}^{\infty}\sum_{j=0}^{k}\frac{1}{k!}\binom{k}{j}\lambda_1^{k-j}\lambda_2^j(1-\mathrm{i}\beta_1t)^{j-k-\alpha_1}(1-\mathrm{i}\beta_2t)^{-j-\alpha_2}. \nonumber
\end{align}
An application of the inversion theorem gives
\begin{align*}
	f_{X_1+X_2}(x)&=\frac{1}{2\pi}\exp\big(-(\lambda_1+\lambda_2)\big)\sum_{k=0}^{\infty}\sum_{j=0}^{k}\frac{1}{k!}\binom{k}{j}\lambda_1^{k-j}\lambda_2^{j}\beta_1^{j-k-\alpha_1}\beta_2^{-j-\alpha_2}\nonumber\\
	&\quad\times\int_{-\infty}^{\infty}\mathrm{e}^{-\mathrm{i}tx}\bigg(\frac{1}{\beta_1}-\mathrm{i}t\bigg)^{j-k-\alpha_1}\bigg(\frac{1}{\beta_2}-\mathrm{i}t\bigg)^{-j-\alpha_2}\,\mathrm{d}t, \quad x>0,
\end{align*}
and evaluating the integral using the integral formula (\ref{int2}) yields the series representation (\ref{sum}).
\end{proof}

In proving Proposition \ref{prop1}, we will make use of the following lemma, which provides a useful finite series representation for the function $U(-n,b,x)$ when $n$ is a non-negative integer.

\begin{lemma}\label{lem55} Let $n$ be a non-negative integer and let $b\in\mathbb{R}$ and $x\in\mathbb{R}$. Then
\begin{align}\label{ulul}
U(-n,b,x)=x^n\sum_{k=0}^n\frac{(-n)_k(-n-b+1)_k}{k!}(-x)^{-k}.    
\end{align}    
\end{lemma}

\begin{proof}
We begin by noting that $U(-n,b,x)=(-1)^nn!L_n^{(b-1)}(x)$ when $n$ is a non-negative integer (see \cite[Eq.\ 13.6.19]{olver}). We will prove the equality (\ref{ulul}) by proving that the right-hand side of the equality is equal to $(-1)^nn!L_n^{(b-1)}(x)$. Using basic identities involving the Pochhammer symbol we obtain that
\begin{align*}
x^n\sum_{k=0}^n\frac{(-n)_k(-n-b+1)_k}{k!}(-x)^{-k}&=\sum_{k=0}^n\frac{n!}{(n-k)!}\frac{(-n-b+1)_k}{k!}x^{n-k} \\
&=n!\sum_{j=0}^n \frac{(-n-b+1)_{n-j}}{(n-j)!j!}x^j\\
&=(-1)^nn!\sum_{j=0}^n\frac{(b+j)_{n-j}}{(n-j)!j!}(-x)^j \\
&=(-1)^nn! L_n^{(b-1)}(x),
\end{align*}
where in the first step we used that $(-n)_k=(-1)^kn!/(n-k)!$, in the third step we used that $(-n-b+1)_{n-j}=(-1)^{n-j}(b+j)_{n-j}$, and in the final step we recognised the series as a generalized Laguerre polynomial using \cite[Eq.\ 18.5.12]{olver}. This proves the lemma.
\end{proof}

\begin{proof}[Proof of Proposition \ref{prop1}] 1. To ease notation, we let $z=(\lambda_2/\beta_2)({1}/{\beta_1}+{1}/{\beta_2})^{-1}$ and 
	\begin{align*}
		N(x)=\frac{\beta_1^{-n}\beta_2^{-\alpha_2}}{(n-1)!}\bigg(\frac{1}{\beta_1}+\frac{1}{\beta_2}\bigg)^{-\alpha_2}x^{n-1}\mathrm{e}^{-\lambda_2-x/\beta_1}.
	\end{align*}
Then from the infinite series representation (\ref{diff0}) of the PDF of $X_1-X_2$ and formula (\ref{ulul}) of Lemma \ref{lem55} we have that, for $x>0$,
	\begin{align}
		f_{X_1-X_2}(x)= N(x)\sum_{k=0}^{\infty}\frac{z^k}{k!}\sum_{s=0}^{n-1}\frac{(1-n)_s(\alpha_2+k)_s}{s!}\bigg(-x\bigg(\frac{1}{\beta_1}+\frac{1}{\beta_2}\bigg)\bigg)^{-s}. \label{comp1}
			\end{align}
		By interchanging the order of summation we obtain that
\begin{align}
			f_{X_1-X_2}(x)&= N(x)\sum_{s=0}^{n-1}(-1)^s\frac{(1-n)_s}{s!}\bigg(x\bigg(\frac{1}{\beta_1}+\frac{1}{\beta_2}\bigg)\bigg)^{-s}\sum_{k=0}^{\infty}\frac{z^k}{k!}(\alpha_2+k)_s\nonumber\\
			&=N(x)\sum_{s=0}^{\infty}(-1)^s\bigg(x\bigg(\frac{1}{\beta_1}+\frac{1}{\beta_2}\bigg)\bigg)^{-s}\frac{(1-n)_s(\alpha_2)_s}{s!}\sum_{k=0}^{\infty}\frac{(\alpha_2+s)_k}{(\alpha_2)_k}\frac{z^k}{k!}\nonumber\\
			&=N(x)\sum_{s=0}^{n-1}(-1)^s\bigg(x\bigg(\frac{1}{\beta_1}+\frac{1}{\beta_2}\bigg)\bigg)^{-s}\frac{(1-n)_s(\alpha_2)_s}{s!}M(\alpha_2+s,\alpha_2,z)\nonumber\\
			&=N(x)\mathrm{e}^z\sum_{s=0}^{n-1}\frac{h_s}{x^s}\nonumber\\
            &=\frac{\beta_1^{-n}\beta_2^{-\alpha_2}}{(n-1)!}\bigg(\frac{1}{\beta_1}+\frac{1}{\beta_2}\bigg)^{-\alpha_2}x^{n-1}\exp\bigg(-\frac{x}{\beta_1}-\frac{\lambda_2\beta_2}{\beta_1+\beta_2}\bigg)\sum_{s=0}^{n-1}\frac{h_s}{x^s},\label{comp1b}
		\end{align}
		where
\begin{align*}
h_s=(-1)^s\bigg(\frac{1}{\beta_1}+\frac{1}{\beta_2}\bigg)^{-s}\frac{(1-n)_s(\alpha_2)_s}{s!}\mathrm{e}^{-z}M(\alpha_2+s,\alpha_2,z).    
\end{align*}
Here, in the third step, we have evaluated the sum over $k$ by using the hypergeometric series representation of the function $M(a,b,x)$ as given in \cite[Eq.\ 13.2.2]{olver}. Using Kummer's transformation 
\begin{equation}\label{kummerm} M(a,b,x)=\mathrm{e}^xM(b-a,b,-x)
\end{equation}
(see \cite[Eq.\ 13.2.39]{olver}) and the relation 
\begin{equation}\label{smsm}
 (\alpha)_nM(-n,\alpha,x)=n!L_n^{(\alpha-1)}(x), \quad n=0,1,2,\ldots,   
\end{equation}
(see \cite[Eq.\ 13.6.19]{olver}) we obtain that        
 \begin{align}
h_s&=(-1)^s\bigg(\frac{1}{\beta_1}+\frac{1}{\beta_2}\bigg)^{-s}\frac{(1-n)_s(\alpha_2)_s}{s!}M(-s,\alpha_2,-z) \nonumber\\
&=(-1)^s\bigg(\frac{1}{\beta_1}+\frac{1}{\beta_2}\bigg)^{-s}(1-n)_s L_s^{(\alpha_2-1)}(-z)\nonumber\\
&=(-1)^s\bigg(\frac{1}{\beta_1}+\frac{1}{\beta_2}\bigg)^{-s}(1-n)_s L_s^{(\alpha_2-1)}\bigg(-\frac{\lambda_2\beta_1}{\beta_1+\beta_2}\bigg)=u_s(n,\alpha_2,\beta_1,\beta_2,\lambda_2).\label{comp1c}
 \end{align}       
This completes the proof of the series representation (\ref{rprp}).

\vspace{3mm}

\noindent 2. From the reasoning of Remark \ref{remdiff0}, formula (\ref{rprp2}) follows on replacing $(x,n, \alpha_2,\beta_1,\beta_2,\lambda_2)$ by $(-x,n,\alpha_1,$ $\beta_2,\beta_1,\lambda_1)$ in formula (\ref{rprp}).
\end{proof}

\begin{proof}[Proof of Theorem \ref{thm2.3}] In the case $\lambda_1=\lambda_2=\lambda$ and $\beta_1=\beta_2=\beta$, the characteristic function (\ref{cfdiff}) reduces to 
	\begin{align*}
		\phi_{X_1-X_2}(t)&=\frac{1}{(1+\beta^2t^2)^{\alpha_1+\alpha_2}}\exp\bigg(-2\lambda+\frac{2\lambda}{1+\beta^2t^2}\bigg)=\mathrm{e}^{-2\lambda}\sum_{k=0}^{\infty}\frac{(2\lambda)^k}{k!(1+\beta^2t^2)^{\alpha_1+\alpha_2+k}}.	
	\end{align*}
	Then, by the inversion theorem, the PDF of $X_1-X_2$ can be expressed as 
	\begin{align*}
		f_{X_1-X_2}(x)&=\frac{1}{2\pi}\mathrm{e}^{-2\lambda}\sum_{k=0}^{\infty}\frac{1}{k!}(2\lambda)^k\int_{-\infty}^{\infty}\frac{\mathrm{e}^{-\mathrm{i}xt}}{(1+\beta^2t^2)^{\alpha_1+\alpha_2+k}}\,\mathrm{d}t, \quad x\in\mathbb{R}.
	\end{align*}
	Evaluating the integral using the integral formula \eqref{int3} and then applying the standard relation $W_{0,\nu}(x)=\sqrt{x/\pi}K_\nu(x/2)$ (see \cite[Eq.\ 10.39.8]{olver}) yields the series representation (\ref{thm2.31}). 
\end{proof}

\begin{proof}[Proof of Theorem \ref{thm4.1}.] 1. Since the random variables $X_1$ and $X_2$ are independent, by the convolution theorem, we have that, for $x>0$, 
\begin{align}\label{convo1}
f_{X_1-X_2}(x)=\int_0^\infty f_{X_1}(y+x)f_{X_2}(y)\,\mathrm{d}y   
\end{align}
and
\begin{align}\label{convo2}
f_{X_1+X_2}(x)=\int_0^x f_{X_1}(x-y)f_{X_2}(y)\,\mathrm{d}y.   
\end{align}
Plugging the expression (\ref{pdfncg}) for the PDFs of $X_1$ and $X_2$ into the integrals (\ref{convo1}) and (\ref{convo2}) and then making the substitution $y=xt$ now yields the integral representations  (\ref{diffint1}) and (\ref{sumint1}).

\vspace{3mm}

\noindent 2. \& 3. The integral representations given in parts 2 and 3 of the theorem are obtained similarly, but we now obtain the simpler expressions by using that $f_{X_1}(x)=\beta_1^{-\alpha_1}x^{\alpha_1-1}\mathrm{e}^{-x/\beta_1}/\Gamma(\alpha_1)$, $x>0$, if $\lambda_1=0$ and $f_{X_2}(x)=\beta_2^{-\alpha_2}x^{\alpha_2-1}\mathrm{e}^{-x/\beta_2}/\Gamma(\alpha_2)$, $x>0$, if $\lambda_2=0$, since in these cases $X_1$ and $X_2$ follow the (central) gamma distribution.

\vspace{3mm}

\noindent 4. Following Remark \ref{remdiff0}, the assertion in part 4 of the theorem follows since $X_1-X_2=_d -(X_1'-X_2')$, where $X_1'\sim\Gamma(\alpha_2,\beta_2,\lambda_2)$ and $X_2'\sim\Gamma(\alpha_1,\beta_1,\lambda_1)$ are independent.
	\end{proof}
	
	\begin{proof}[Proof of Theorem \ref{thm4.2}.] By applying the following integral representation of the modified Bessel function of the second kind
\begin{equation*}
	K_\nu(x)=\frac{1}{2}\bigg(\frac{x}{2}\bigg)^\nu\int_{0}^{\infty}t^{-\nu-1}\exp\bigg(-t-\frac{x^2}{4t}\bigg)\,\mathrm{d}t, \quad \nu\in\mathbb{R}
,\: x>0
\end{equation*} 
(see \cite[Eq.\ 10.32.10]{olver}) to the series representation (\ref{thm2.31})
followed by an interchange of the order of summation and integration we have that
		\begin{align}
			f_{X_1-X_2}(x)&=\frac{\mathrm{e}^{-2\lambda}}{2\beta\sqrt{\pi}}\bigg(\frac{|x|}{2\beta}\bigg)^{2(\alpha_1+\alpha_2-1/2)}\nonumber\\
			&\quad\times\int_{0}^{\infty}t^{-(\alpha_1+\alpha_2+1/2)}\exp\bigg(-t-\frac{x^2}{4\beta^2t}\bigg)\sum_{k=0}^{\infty}\frac{1}{k!\Gamma(\alpha_1+\alpha_2+k)}\bigg(\frac{\lambda x^2}{2\beta^2t}\bigg)^k\,\mathrm{d}t. \label{eee1}
		\end{align}
By comparison to the series representation of the modified Bessel function of the first kind (see \cite[Eq.\ 10.25.2]{olver}) we have that 
\begin{equation}
\sum_{k=0}^{\infty}\frac{1}{k!\Gamma(\alpha_1+\alpha_2+k)}\bigg(\frac{\lambda x^2}{2\beta^2t}\bigg)^k=\bigg(\frac{1}{\beta}\sqrt{\frac{\lambda}{2t}}|x|\bigg)^{1-\alpha_1-\alpha_2}I_{\alpha_1+\alpha_2-1}\bigg(\frac{1}{\beta}\sqrt{\frac{2\lambda}{t}}|x|\bigg),    \label{eee2}
\end{equation}
and by combining the expressions (\ref{eee1}) and (\ref{eee2}) we obtain the integral representation (\ref{thm4.21}). 
\end{proof}

In proving Propositions \ref{prop2.4} and \ref{prop2.5} we will make use of the following limiting forms for the confluent hypergeometric function of the second kind (see \cite[Section 13.2(iii)]{olver}). For $a=-n$ or $a=-n+b-1$, where $n$ is a non-negative integer, the following limiting forms hold. As $x\rightarrow0$,
\begin{align}
 U(-n,b,x)&\sim (-1)^n(b)_n, \label{U04}\\
U(-n+b-1,b,x)&\sim (-1)^n(2-b)_nx^{1-b}.
\end{align}
Otherwise, the following limiting forms hold. As $x\rightarrow0$,
\begin{align}
	U(a,b,x)&\sim\frac{\Gamma(1-b)}{\Gamma(a-b+1)},\quad b<1,\label{U02}\\
    	U(a,1,x)&\sim-\frac{\ln|x|}{\Gamma(a)}, \label{U01} \\
        U(a,b,x)&\sim\frac{\Gamma(b-1)}{\Gamma(a)}x^{1-b}, \quad b\geq1. \label{U03} 
\end{align}

\begin{proof}[Proof of Proposition \ref{prop2.4}] 1. Suppose that $\alpha_1+\alpha_2<1$. From the limiting forms (\ref{U04})--(\ref{U02}) and (\ref{U03}) it can be seen that the term in the double infinite series (\ref{diff}) with $j=k=0$ is of order $|x|^{\alpha_1+\alpha_2-1}$ as $x\rightarrow0$ and that in this limit all other terms are of lower order. Thus, as $x\rightarrow0$,
\begin{align}\label{hbb}
f_{X_1-X_2}(x)\sim \beta_1^{-\alpha_1}\beta_2^{-\alpha_2}\mathrm{e}^{-\lambda_1-\lambda_2}\frac{\Gamma(1-\alpha_1-\alpha_2)}{\Gamma(a_{0,0}(x))\Gamma(1-a_{0,0}(x))}\,|x|^{\alpha_1+\alpha_2-1}.   
\end{align}
The limiting form (\ref{diff0lim1}) is now obtained by applying the reflection formula $\Gamma(z)\Gamma(1-z)=\pi/\sin(\pi z)$ (see \cite[Eq.\ 5.5.3]{olver}) to the limiting form (\ref{hbb}).

\vspace{3mm}
    
\noindent 2. Suppose that $\alpha_1+\alpha_2=1$. By a similar argument to the one we employed in part 1 of the proof, the dominate term in the double infinite series (\ref{diff}) is the one for which $j=k=0$. On applying the limiting form (\ref{U01}) to the term with $j=k=0$ followed by an application of the reflection formula $\Gamma(z)\Gamma(1-z)=\pi/\sin(\pi z)$ we obtain the limiting form (\ref{diff0lim2}).

\vspace{3mm}
    
\noindent 3. Suppose that $\alpha_1+\alpha_2\leq1$. From the limiting forms (\ref{diff0lim1}) and (\ref{diff0lim2}) it follows that the PDF of the difference $X_1-X_2$ has a singularity at $x=0$. Now, since the function $U(a,b,x)$ is bounded for $x\not=0$, it follows that the PDF $f_{X_1-X_2}(x)$ is bounded for $x\not=0$, meaning that the distribution of $X_1-X_2$ has a unique mode at $x=0$.

\vspace{3mm}

\noindent 4. Suppose that $\alpha_1+\alpha_2>1$. Taking the limit $x\downarrow0$ in the series representation (\ref{diff}) for the PDF of $X_1-X_2$ using the limiting form (\ref{U02}) gives that
\begin{align*}
f_{X_1-X_2}(0)&=\beta_1^{-\alpha_1}\beta_2^{-\alpha_2}\mathrm{e}^{-\lambda_1-\lambda_2}\\
&\quad\times\sum_{k=0}^\infty\sum_{j=0}^k\frac{\Gamma(\alpha_1+\alpha_2-1+k)}{k!\Gamma(\alpha_1+k-j)\Gamma(\alpha_2+j)}\binom{k}{j}\bigg(\frac{\lambda_1}{\beta_1}\bigg)^k\bigg(\frac{\lambda_2\beta_1}{\lambda_1\beta_2}\bigg)^j\bigg(\frac{1}{\beta_1}+\frac{1}{\beta_2}\bigg)^{1-\alpha_1-\alpha_2-k}.    
\end{align*}
On making the change of index $m=k-j$ and then simplifying the resulting expression by writing out the binomial coefficient in terms of factorials and using the basic relation $\Gamma(a+n)=(a)_n \,\Gamma(a)$ we obtain the expression
\begin{align*}
f_{X_1-X_2}(0)&=\mathrm{e}^{-\lambda_1-\lambda_2} \beta_1^{-\alpha_1}\beta_2^{-\alpha_2}\bigg(\frac{1}{\beta_1}+\frac{1}{\beta_2}\bigg)^{1-\alpha_1-\alpha_2} \frac{\Gamma(\alpha_1+\alpha_2-1)}{\Gamma(\alpha_1)\Gamma(\alpha_2)}\\
&\quad\times \sum_{m=0}^\infty\sum_{j=0}^\infty \frac{(\alpha_1+\alpha_2-1)_{m+j}}{(\alpha_1)_m(\alpha_2)_j}\frac{1}{m!j!}\bigg(\frac{\lambda_1\beta_2}{\beta_1+\beta_2}\bigg)^m\bigg(\frac{\lambda_2\beta_1}{\beta_1+\beta_2}\bigg)^j,
\end{align*}
and by comparison to the power series representation (\ref{horn}) of the Horn function $\Psi_2$ we obtain the desired formula (\ref{hornf}).

\vspace{3mm}
    
\noindent 5. Suppose that $\alpha_1+\alpha_2>1$. From part 4, we know that the PDF of $X_1-X_2$ is bounded at $x=0$, and since the function $U(a,b,x)$ is bounded for $x\not=0$ it follows that the PDF $f_{X_1-X_2}(x)$ is bounded for all $x\in\mathbb{R}$.
\end{proof}

\begin{proof}[Proof of Proposition \ref{prop2.5}] 1. Since $\lim_{x\rightarrow0}M(a,b,x)=1$ (see \cite[Eq.\ 13.2.13]{olver}), it follows that as $x\rightarrow0$ the dominate term in the double infinite series (\ref{sum}) is the term for which $j=k=0$. From this observation we thus arrive at the limiting form (\ref{lim7}).

\vspace{3mm}
    
\noindent 2.\ \& 3. The proofs of parts 2 and 3 are similar to the proofs of parts 3 and 4 of Proposition \ref{prop2.4}, and we omit the details.
\end{proof}    

\begin{proof}[Proof of Theorem \ref{thm3.1}] 1. Suppose that $\lambda_1\not=0$. We begin by considering the PDF of the difference $X_1-X_2$. By applying the asymptotic expansion
\begin{equation}\label{asyu}
U(a,b,x)\sim x^{-a}\sum_{s=0}^{\infty}\frac{(a)_s(a-b+1)_s}{s!}(-x)^{-s}, \quad x\rightarrow\infty,    
\end{equation}
(see \cite[Eq.\ 13.7.1]{olver}) to the series representation (\ref{diff}) (with a re-indexing $m=k-j$) of the PDF of $X_1-X_2$ we obtain that, as $x\rightarrow\infty$,
		\begin{align}\label{iv1}
			f_{X_1-X_2}(x)\sim N_-(x)\sum_{m=0}^{\infty}\sum_{j=0}^\infty\frac{z_-^j(yx)^m}{j!m!(\alpha_1)_m}\sum_{s=0}^{\infty}(-1)^s\frac{(1-\alpha_1-m)_s(\alpha_2+j)_s}{s!}(w_-x)^{-s},
		\end{align}
		where $w_-=1/\beta_1+1/\beta_2$, $y=\lambda_1/\beta_1$, $z_-=(\lambda_2/\beta_2)({1}/{\beta_1}+{1}/{\beta_2})^{-1}$ and 
		\begin{align*}
        N_-(x)=\frac{\beta_1^{-\alpha_1}\beta_2^{-\alpha_2}}{\Gamma(\alpha_1)}\bigg(\frac{1}{\beta_1}+\frac{1}{\beta_2}\bigg)^{-\alpha_2}x^{\alpha_1-1}\mathrm{e}^{-\lambda_1-\lambda_2-x/\beta_1}.
		\end{align*}
        
We now obtain a similar intermediate asymptotic expansion for the PDF of the sum $X_1+X_2$. By Kummer's transformation (\ref{kummerm}) and the following asymptotic expansion
\begin{equation}\label{asym}
M(a,b,x)\sim\frac{\mathrm{e}^{x}x^{a-b}\Gamma(b)}{\Gamma(a)}\sum_{s=0}^{\infty}\frac{(1-a)_s(b-a)_s}{s!}x^{-s}, \quad x\rightarrow\infty, \:a\notin\{0,-1,-2,\ldots\},    
\end{equation}
(see \cite[Eq.\ 13.7.1]{olver}) we get that, as $x\rightarrow\infty$,
\begin{align}
&M\bigg(\alpha_2+j,\alpha_1+\alpha_2+m+j,\bigg(\frac{1}{\beta_1}-\frac{1}{\beta_2}\bigg)x\bigg)\nonumber\\
&\quad=\exp\bigg\{\bigg(\frac{1}{\beta_1}-\frac{1}{\beta_2}\bigg)x\bigg\}M\bigg(\alpha_1+m,\alpha_1+\alpha_2+m+j,\bigg(\frac{1}{\beta_2}-\frac{1}{\beta_1}\bigg)x\bigg) \nonumber \\
&\quad\sim \frac{\Gamma(\alpha_1+\alpha_2+m+j)}{\Gamma(\alpha_1+m)}\bigg(\bigg(\frac{1}{\beta_2}-\frac{1}{\beta_1}\bigg)x\bigg)^{-\alpha_2-j}\sum_{s=0}^{\infty}\frac{(1-\alpha_1-m)_s(\alpha_2+j)_s}{s!}\bigg(\bigg(\frac{1}{\beta_2}-\frac{1}{\beta_1}\bigg)x\bigg)^{-s}, \label{eswdxw}
\end{align}
where we were able to apply the asymptotic expansion (\ref{asym}) since $\beta_1>\beta_2$. Applying the asymptotic expansion (\ref{eswdxw}) to the series representation (\ref{sum}) of the PDF of $X_1+X_2$ gives that, as $x\rightarrow\infty$,
\begin{align}\label{iv2}
			f_{X_1+X_2}(x)\sim N_+(x)\sum_{m=0}^{\infty}\sum_{j=0}^\infty\frac{z_+^j(yx)^m}{j!m!(\alpha_1)_m}\sum_{s=0}^{\infty}(-1)^s\frac{(1-\alpha_1-m)_s(\alpha_2+j)_s}{s!}(w_+x)^{-s},
			\end{align}
where $w_+=1/\beta_1-1/\beta_2$, $z_+=(\lambda_2/\beta_2)({1}/{\beta_2}-{1}/{\beta_1})^{-1}$ and
\begin{align*}
N_+(x)=\frac{\beta_1^{-\alpha_1}\beta_2^{-\alpha_2}}{\Gamma(\alpha_1)}\bigg(\frac{1}{\beta_2}-\frac{1}{\beta_1}\bigg)^{-\alpha_2}x^{\alpha_1-1}\mathrm{e}^{-\lambda_1-\lambda_2-x/\beta_1}.
\end{align*}

From (\ref{iv1}) and (\ref{iv2}) we obtain the unified intermediate asymptotic expansion: as $x\rightarrow\infty$,
\begin{align*}
			f_{X_1\mp X_2}(x)&\sim N_\mp(x)\sum_{m=0}^{\infty}\sum_{j=0}^\infty\frac{z_{\mp}^j(yx)^m}{j!m!(\alpha_1)_m}\sum_{s=0}^{\infty}(-1)^s\frac{(1-\alpha_1-m)_s(\alpha_2+j)_s}{s!}(w_\mp x)^{-s}\\
            &=N_{\mp}(x)\sum_{m=0}^{\infty}\sum_{j=0}^\infty\sum_{s=0}^\infty (-1)^s\frac{(1-\alpha_1)_s(\alpha_2)_s(\alpha_2+s)_j}{(\alpha_1-s)_m(\alpha_2)_j}\frac{z_{\mp}^j(yx)^m}{j!m!s!}(w_{\mp}x)^{-s},
			\end{align*}   
where in obtaining the equality we used the relations
\begin{align*}
(1-\alpha_1-m)_s=\frac{(1-\alpha_1)_s(\alpha_1)_m}{(\alpha_1-s)_m}, \quad (\alpha_2+j)_s=\frac{(\alpha_2)_s}{(\alpha_2)_j}(\alpha_2+s)_j.    
\end{align*}            
Now, interchanging the order of summation gives that, as $x\rightarrow\infty$,
		\begin{align}
			f_{X_1\mp X_2}(x) \sim N_\mp(x)\sum_{s=0}^{\infty}\frac{(-1)^s(\alpha_2)_s(1-\alpha_1)_s}{s!}(w_{\mp}x)^{-s}\sum_{m=0}^{\infty}\frac{(yx)^m}{m!(\alpha_1-s)_m}\sum_{j=0}^\infty\frac{(\alpha_2+s)_j}{(\alpha_2)_j}\frac{z_{\mp}^j}{j!}. \label{plug}
            \end{align}

We now evaluate the infinite series over the index $j$:
\begin{align}\label{plug1}
\sum_{j=0}^\infty\frac{(\alpha_2+s)_j}{(\alpha_2)_j}\frac{z_{\mp}^j}{j!}=M(\alpha_2+s,\alpha_2,z_{\mp})   =\mathrm{e}^{z_\mp}M(-s,\alpha_2,-z_{\mp})=\mathrm{e}^{z_\mp}\frac{s!}{(\alpha_2)_s}L_s^{(\alpha_2-1)}(-z_{\mp}),
\end{align}
where we obtained the second equality using Kummer's transformation (\ref{kummerm}) and obtained the third equality using the relation (\ref{smsm}). We now move on to evaluate the infinite series over the index $m$:
\begin{align}
\sum_{m=0}^{\infty}\frac{(yx)^m}{m!(\alpha_1-s)_m}&={}_0F_1(-;\alpha_1-s;yx)\nonumber\\
&=\Gamma(\alpha_1-s)\,(yx)^{(1-\alpha_1+s)/2}I_{\alpha_1-s-1}\big(2\sqrt{yx}\big) \nonumber  \\
&=\frac{(-1)^s\Gamma(\alpha_1)}{(1-\alpha_1)_s}(yx)^{(1-\alpha_1+s)/2}I_{\alpha_1-s-1}\big(2\sqrt{yx}\big), \label{plug2}
\end{align}
where we used the relation ${}_0F_1(-;b;x)=\Gamma(b)x^{(1-b)/2}I_{b-1}(2\sqrt{x})$ (see \cite[Eq.\ 10.39.9]{olver}), and applied the basic relation $(1-a)_n\Gamma(a-n)=(-1)^n\Gamma(a)$ in the last step. Plugging the evaluations (\ref{plug1}) and (\ref{plug2}) into (\ref{plug}) gives that, as $x\rightarrow\infty$,
\begin{align}
f_{X_1\mp X_2}(x)\sim \Gamma(\alpha_1)\mathrm{e}^{z_{\mp}}N_{\mp}(x)\sum_{s=0}^{\infty}L_s^{(\alpha_2-1)}(-z_{\mp})(w_{\mp}x)^{-s}(yx)^{(1-\alpha_1+s)/2}I_{\alpha_1-s-1}\big(2\sqrt{yx}\big).   \label{fdiff1}
\end{align}
We will now make use of the following asymptotic expansion of the modified Bessel function of the first kind:
\begin{align}
\label{Iab}
	I_\nu(x)&\sim\frac{\mathrm{e}^x}{(2\pi x)^{1/2}}\sum_{k=0}^{\infty}\frac{(1/2+\nu)_k(1/2-\nu)_k}{k!\,(2x)^k}, \quad x\rightarrow\infty,\: \nu\in\mathbb{R},
\end{align}
(see \cite[Eq.\ 10.40.1]{olver}). Applying the asymptotic expansion (\ref{Iab}) to \eqref{fdiff1} now gives that, as $x\rightarrow\infty$,
		\begin{align}\label{fdiff2}
			f_{X_1\mp X_2}(x)&\sim \tilde{N}_{\mp}(x)\sum_{s=0}^{\infty} y^{s/2}L_s^{(\alpha_2-1)}(z_\mp)\,(w_\mp^2x)^{-s/2}\sum_{k=0}^{\infty}\frac{(3/2+s-\alpha_1)_k(\alpha_1-s-1/2)_k}{4^k\!k!(yx)^{k/2}},
		\end{align}
		where
		\begin{align*}
			\tilde{N}_{\mp}(x)&=\frac{\Gamma(\alpha_1)}{2\sqrt{\pi}}\mathrm{e}^{z\mp}N_{\mp}(x)\,(yx)^{(1-2\alpha_1)/4}\mathrm{e}^{2\sqrt{yx}}\\
            &=\frac{1}{2\sqrt{\pi}}\bigg(\frac{1}{\beta_2}\pm\frac{1}{\beta_1}\bigg)^{-\alpha_2}\beta_1^{-\alpha_1}\beta_2^{-\alpha_2}\bigg(\frac{\lambda_1}{\beta_1}\bigg)^{(1-2\alpha_1)/4}\exp\bigg(-\lambda_1-\frac{\lambda_2\beta_2}{\beta_2\pm\beta_1}\bigg)\nonumber\\
		&\quad\times x^{(2\alpha_1-3)/4}\exp\bigg(2\sqrt{\frac{\lambda_1x}{\beta_1}}-\frac{x}{\beta_1}\bigg).
		\end{align*}
		Finally, applying the Cauchy product formula to \eqref{fdiff2}, we get the asymptotic expansion
\begin{align*}
f_{X_1\mp X_2}(x)\sim \tilde{N}_{\mp}(x)\sum_{l=0}^\infty \frac{c_l}{x^{l/2}}, \quad x\rightarrow\infty,
\end{align*}        
where, for $l\geq0$,
\begin{align*}
c_l=y^{-l/2}\sum_{j=0}^lL_j^{(\alpha_2-1)}(z_\mp)\frac{(3/2+j-\alpha_1)_{l-j}(\alpha_1-j-1/2)_{l-j}}{(l-j)!4^{l-j}}\bigg(\frac{y}{w_\mp}\bigg)^{j},    
\end{align*}
which is the desired asymptotic expansion \eqref{thm3.11} with coefficients $c_l$, $l\geq1$, as given by the expression (\ref{c1for}).
	
\vspace{3mm}

\noindent 2. Suppose that $\lambda_1=0$. We begin by considering the asymptotic expansion for the PDF of $X_1-X_2$. Applying the asymptotic expansion (\ref{asyu}) to the confluent hypergeometric function of the second kind in the series representation (\ref{diff0}) of the PDF of $X_1-X_2$ gives that
	\begin{align}
		f_{X_1-X_2}(x)\sim N_-(x)\sum_{k=0}^{\infty}\frac{z_-^k}{k!}\sum_{s=0}^{\infty}\frac{(1-\alpha_1)_s(\alpha_2+k)_s}{s!}\bigg(-\bigg(\frac{1}{\beta_1}+\frac{1}{\beta_2}\bigg)x\bigg)^{-s}, \quad x\rightarrow\infty, \label{int177}
			\end{align}
where $z_-=({\lambda_2}/{\beta_2})({1}/{\beta_1}+{1}/{\beta_2})^{-1}$. We now obtain a similar intermediate asymptotic expansion for the PDF of the sum $X_1+X_2$. 
By applying Kummer's transformation (\ref{kummerm}) and the asymptotic expansion (\ref{asym}) we get that, as $x\rightarrow\infty$,
\begin{align}
&M\bigg(\alpha_2+k,\alpha_1+\alpha_2+k,\bigg(\frac{1}{\beta_1}-\frac{1}{\beta_2}\bigg)x\bigg)\nonumber\\
&\quad=\exp\bigg\{\bigg(\frac{1}{\beta_1}-\frac{1}{\beta_2}\bigg)x\bigg\}M\bigg(\alpha_1,\alpha_1+\alpha_2+k,\bigg(\frac{1}{\beta_2}-\frac{1}{\beta_1}\bigg)x\bigg) \nonumber \\
&\quad\sim \frac{\Gamma(\alpha_1+\alpha_2+k)}{\Gamma(\alpha_1)}\bigg(\bigg(\frac{1}{\beta_2}-\frac{1}{\beta_1}\bigg)x\bigg)^{-\alpha_2-k}\sum_{s=0}^{\infty}\frac{(1-\alpha_1)_s(\alpha_2+k)_s}{s!}\bigg(\bigg(\frac{1}{\beta_2}-\frac{1}{\beta_1}\bigg)x\bigg)^{-s}, \label{vjfdvb}
\end{align}
where we were able to apply the asymptotic expansion (\ref{asym}) since $\beta_1>\beta_2$. Applying the asymptotic expansion (\ref{vjfdvb}) to the series representation (\ref{sum2}) of the PDF of $X_1+X_2$ now gives that
\begin{align} \label{int277}
			f_{X_1+X_2}(x)\sim N_+(x)\sum_{k=0}^{\infty}\frac{z_+^k}{k!}\sum_{s=0}^{\infty}\frac{(1-\alpha_1)_s(\alpha_2+k)_s}{s!}\bigg(\bigg(\frac{1}{\beta_2}-\frac{1}{\beta_1}\bigg)x\bigg)^{-s}, \quad x\rightarrow\infty,
		\end{align}
where $z_+=(\lambda_2/\beta_2)(1/\beta_2-1/\beta_1)^{-1}$.        
From the asymptotic expansions (\ref{int177}) and (\ref{int277}) we obtain the following unified asymptotic expansion for $X_1\mp X_2$:
\begin{align}\label{comp2}
f_{X_1\mp X_2}(x)\sim N_\mp(x)\sum_{k=0}^{\infty}\frac{z_\mp^k}{k!}\sum_{s=0}^{\infty}\frac{(1-\alpha_1)_s(\alpha_2+k)_s}{s!}\bigg(-\bigg(\frac{1}{\beta_1}\pm\frac{1}{\beta_2}\bigg)x\bigg)^{-s}, \quad x\rightarrow\infty.    
\end{align}
By comparison between the asymptotic expansion (\ref{comp2}) and equation (\ref{comp1}) (with the coefficients in the series given by (\ref{comp1c})) we see that applying the same manipulations used to obtain the equality (\ref{comp1b}) from (\ref{comp1}) results in the asymptotic expansion		
\begin{align*}
f_{X_1\mp X_2}(x)\sim  \frac{\beta_1^{-\alpha_1}\beta_2^{-\alpha_2}}{\Gamma(\alpha_1)}\bigg(\frac{1}{\beta_2}\pm\frac{1}{\beta_1}\bigg)^{-\alpha_2}x^{\alpha_1-1}\exp\bigg(-\frac{x}{\beta_1}-\frac{\lambda_2\beta_2}{\beta_2\pm\beta_1}\bigg)\sum_{s=0}^{\infty}\frac{H_s}{x^s},  \quad x\rightarrow\infty,
\end{align*}	
where	
\begin{align*}
H_s=(-1)^s(1-\alpha_1)_s\bigg(\frac{1}{\beta_1}\pm\frac{1}{\beta_2}\bigg)^{-s}L_s^{(\alpha_2-1)}\bigg(-\frac{\lambda_2\beta_1}{\beta_1\pm\beta_2}\bigg),    
\end{align*}
which is the desired asymptotic expansion (\ref{thm3.12}).

\vspace{3mm}
        
\noindent 3 \& 4.	Following the reasoning given in Remark \ref{remdiff0}, the asymptotic expansion (\ref{thm3.11aa}) follows on replacing $(x,\alpha_1, \alpha_2,\beta_1,\beta_2,\lambda_1,\lambda_2)$ by $(-x,\alpha_2,\alpha_1,$ $\beta_2,\beta_1,\lambda_2,\lambda_1)$ in the asymptotic expansion (\ref{thm3.11}) for the PDF of $X_1-X_2$, whilst
the asymptotic expansion (\ref{thm3.14}) follows on replacing $(x,\alpha_1, \alpha_2,\beta_1,\beta_2,\lambda_2)$ by $(-x,\alpha_2,\alpha_1,$ $\beta_2,\beta_1,\lambda_1)$ in the asymptotic expansion (\ref{thm3.12}) for the PDF of $X_1-X_2$.
	\end{proof}  

In proving Theorems \ref{thm3.5} and \ref{thm3.7}, we will make use of two lemmas from \cite{gz25}, which we now recall.     

\begin{lemma}[Lemma 5.1 of \cite{gz25}] 
	\label{lemma5.1}
	1. Fix $a>0$, $b\in\mathbb{R}$, $m\in\mathbb{R}$ and let $u_0, u_1, \ldots $ be real-valued constants. Then, as $x\rightarrow\infty$,
	\begin{align}\nonumber
		\int_{x}^{\infty}t^m\exp\big(-at+b\sqrt{t}\big)\sum_{l=0}^{\infty}\frac{u_l}{t^{l/2}}\,\mathrm{d}t\sim\frac{x^{m}}{a}\exp\big(-ax+b\sqrt{x}\big)\sum_{p=0}^{\infty}\frac{U_p}{x^{p/2}},
	\end{align}
	where
	\begin{align*}
		U_p&=\sum_{\substack{i,j,k,l\geq0 \\ i+2j+k+l=p}}(-1)^{i+j}u_l\binom{2m+1-l}{k}\binom{2m-l-k-2j}{i}\frac{((l+k)/2-m)_j}{a^j}\bigg(\frac{b}{2a}\bigg)^{k+i}, \quad p\geq0.
	\end{align*}
	2. Fix $a>0$, $m\in\mathbb{R}$ and let $v_0,v_1,\ldots$ be real-valued constants. Then, as $x\rightarrow\infty$,
	\begin{equation}
		\int_{x}^{\infty}t^m\mathrm{e}^{-at}\sum_{j=0}^{\infty}\frac{v_j}{t^j}\,\mathrm{d}t\sim\frac{1}{a}x^m\mathrm{e}^{-ax}\sum_{k=0}^{\infty}\frac{V_k}{x^k}, \nonumber
	\end{equation}
	where $V_k=\sum_{j=0}^{k}v_j(k-m)_{k-j}/(-a)^{k-j}$ for $k\geq0$.
\end{lemma}	 

\begin{lemma}[Lemma 5.2 of \cite{gz25}]\label{lem5.2}
	Fix $a,A,z>0$ and $b,m\in\mathbb{R}$, and Let $g:(0,\infty)\rightarrow\mathbb{R}$ be such that $g(x)=O(x^{-1/2})$ as $x\rightarrow\infty$.  Then the equation
	\begin{equation}\label{l5.21}
		Ax^m\exp\big(-ax+b\sqrt{x}\big)(1+g(x))=z,
	\end{equation}
	has a unique solution $x_*$ for sufficiently small $z$, and this solution satisfies
	\begin{align}\label{l5.22}
		x_*&=\frac{1}{a}\bigg\{\ln(1/z)+\frac{b}{\sqrt{a}}\sqrt{\ln(1/z)}+m\ln(\ln(1/z))+\frac{b^2}{2a}+\ln\bigg(\frac{A}{a^m}\bigg)\nonumber\\
		&\quad+\frac{bm}{2\sqrt{a}}\frac{\ln(\ln(1/z))}{\sqrt{\ln(1/z)}}\bigg\}+O\bigg(\frac{1}{\sqrt{\ln(1/z)}}\bigg), \quad z\rightarrow0.
	\end{align}
	When $b=0$ and $g(x)=O(x^{-1})$ as $x\rightarrow\infty$, the error in  (\ref{l5.22}) is of order $O(1/\ln(1/z))$.
\end{lemma}

	\begin{proof}[Proof of Theorem \ref{thm3.5}] 
Since $\overline{F}_{X_1\mp X_2}(x)=\int_x^\infty f_{X_1\mp X_2}(t)\,\mathrm{d}t$, the asymptotic expansion (\ref{thm3.61}) follows from applying part 1 of Lemma \ref{lemma5.1} in combination with the asymptotic expansion (\ref{thm3.11}) for the PDF of $X_1\mp X_2$. Likewise, we can obtain the asymptotic expansion (\ref{thm3.52}) by applying part 2 of Lemma \ref{lemma5.1} together with the asymptotic expansion (\ref{thm3.12}). Since $F_{X_1-X_2}(x)=\int_{-\infty}^x f_{X_1-X_2}(t)\,\mathrm{d}t$, we can derive the asymptotic expansions (\ref{negtail}) and (\ref{thm3.54}) via a similar procedure, but this time making use of the asymptotic expansions (\ref{thm3.11aa}) and (\ref{thm3.14}) for the left-tail asymptotics of the PDF of $X_1-X_2$. 
		\end{proof} 

	\begin{proof}[Proof of Theorem \ref{thm3.7}] 1. Suppose $\lambda_1\neq0$. Since $\overline{F}_{X_1\mp X_2}(Q_{X_1\mp X_2}(p))=1-p$, it follows from the asymptotic expansion (\ref{thm3.61}) that $Q_{X_1\mp X_2}(p)$ solves an equation of the form (\ref{l5.21}) from Lemma \ref{lem5.2} with
		\begin{align*}
			A&=\frac{1}{2\sqrt{\pi}}\bigg(\frac{1}{\beta_2}\pm\frac{1}{\beta_1}\bigg)^{-\alpha_2}\beta_1^{1-\alpha_1}\beta_2^{-\alpha_2}\bigg(\frac{\lambda_1}{\beta_1}\bigg)^{(1-2\alpha_1)/4}\exp\bigg(-\lambda_1-\frac{\lambda_2\beta_2}{\beta_2\pm\beta_1}\bigg),\nonumber\\
			 a&=1/\beta_1, \quad b=2\sqrt{\lambda_1/\beta_1},\quad m=(2\alpha_1-3)/4, \quad z=1-p.
		\end{align*}
		We thus obtain the asymptotic approximation (\ref{thm3.71}) for the quantile function $Q_{X_1\mp X_2}(p)$ as $p\rightarrow1$ by substituting the above values for $A, a, b,m$ and $z$ into the asymptotic approximation (\ref{l5.22}).

        \vspace{3mm}
		
	\noindent	 2. Suppose now  that $\lambda_1=0$. Then arguing as in part 1 of the proof but this time making use of the asymptotic expansion (\ref{thm3.52}) for $\overline{F}_{X_1\mp X_2}(x)$, we see that the quantile function $Q_{X_1\mp X_2}(p)$ solves an equation of the form (\ref{l5.21}) with 
		\begin{align*}
			A&=\frac{\beta_1^{1-\alpha_1}\beta_2^{-\alpha_2}}{\Gamma(\alpha_1)}\bigg(\frac{1}{\beta_2}\pm\frac{1}{\beta_1}\bigg)^{-\alpha_2}\exp\bigg(-\frac{\lambda_2\beta_2}{\beta_2\pm \beta_1}\bigg),\nonumber\\
			a&=1/\beta_1,\quad b=0, \quad m=\alpha_1-1,\quad z=1-p,
		\end{align*}
        and on applying the approximation (\ref{l5.22}) with these values we obtain the asymptotic approximation (\ref{thm3.72}).

        \vspace{3mm}
		
	\noindent 	3 \& 4. This is similar to the proofs of parts 1 and 2, but this time we use that $Q_{X_1-X_2}(p)$ is a solution of the equation $F_{X_1-X_2}(Q_{X_1-X_2}(p))=p$ and apply the asymptotic expansions (\ref{negtail}) (when $\lambda_2\not=0$) and (\ref{thm3.54}) (when $\lambda_2=0$).
		\end{proof}

\begin{proof}[Proof of Corollary \ref{corpn2}] 1. Suppose that $r_X+r_Y\not=0$. We obtain the asymptotic expansion (\ref{33a}) with simplified coefficients (\ref{33b}) by setting $n=1$ in the asymptotic expansion (\ref{expansion11}) and simplifying as follows. Firstly, we express the Laguerre polynomials in terms of the physicist's Hermite polynomials via the relation
\begin{align}\label{lhlh}
L_k^{(-1/2)}(x)=\frac{(-1)^k}{4^{k}\,k!}H_{2k}\big(\sqrt{x}\big),    
\end{align}
which follows from combining equations 18.11.2 and 13.6.16
of \cite{olver} and using the basic formula $(1/2)_k=(2k!)/(4^k\,k!)$. 
We then simplify the Pochammer symbols by using the basic formulas
\begin{align*}
(j+1)_{l-j}=\frac{l!}{j!}; \quad (-j)_{l-j}=(-1)^{l-j}\frac{j!}{(2j-l)!}, \: l\leq 2j; \quad (-j)_{l-j}=0, \: l>2j.    
\end{align*}
The latter relation $(-j)_{l-j}=0$ for $l>2j$ implies that the summands with $j< \lceil l/2\rceil$ vanish from the series, leaving just those for which $ \lceil l/2\rceil\leq j\leq l$. After a simple calculation that employs these considerations, we arrive at the desired asymptotic expansion (\ref{33a}).

\vspace{3mm}

\noindent 2. Suppose that $r_X+r_Y=0$. In this case, we obtain the desired simplification for the coefficients $d_k(\rho_X,\rho,1)$, $k\geq1$, by again using the relation (\ref{lhlh}) as well as the formula $(1/2)_k=(2k!)/(4^k\,k!)$.
\end{proof}          

\section*{Acknowledgements} Both authors are funded by EPSRC grant EP/Y008650/1.

\end{document}